\documentclass{article}%
\usepackage{amsfonts}
\usepackage{graphicx}
\usepackage{amsmath}%
\usepackage{amssymb}
\providecommand{\U}[1]{\protect\rule{.1in}{.1in}}
\providecommand{\U}[1]{\protect\rule{.1in}{.1in}}
\providecommand{\U}[1]{\protect\rule{.1in}{.1in}}
\newtheorem{theorem}{Theorem}

\newtheorem{corollary}[theorem]{Corollary}

\newtheorem{definition}[theorem]{Definition}

\newtheorem{lemma}[theorem]{Lemma}
\newtheorem{notation}[theorem]{Notation}

\newtheorem{proposition}[theorem]{Proposition}
\newtheorem{remark}[theorem]{Remark}

\newenvironment{proof}[1][Proof]{\textbf{#1.} }{\ \rule{0.5em}{0.5em}}
\begin{document}

\title{Differential Geometry of Microlinear Fr\"{o}licher Spaces I}
\author{Hirokazu Nishimura\\Institute of Mathematics, University of Tsukuba\\Tsukuba, Ibaraki, 305-8571, Japan}
\maketitle

\begin{abstract}
The central object of synthetic differential geometry is microlinear spaces.
In our previous paper [Microlinearity in Fr\"{o}licher spaces -beyond the
regnant philosophy of manifolds-, International Journal of Pure and Applied
Mathematics, 60 (2010), 15-24] we have emancipated microlinearity from within
well-adapted models to Fr\"{o}licher spaces. Therein we have shown that
Fr\"{o}licher spaces which are microlinear as well as Weil exponentiable form
a cartesian closed category. To make sure that such Fr\"{o}licher spaces are
the central object of infinite-dimensional differential geometry, we develop
the theory of vector fields on them in this paper. Our principal result is
that all vector fields on such a Fr\"{o}licher space form a Lie algebra.

\end{abstract}

\section{Introduction}

It is often forgotten that Newton, Leibniz, Euler and many other
mathematicians in the 17th and 18th centuries developed differential calculus
and analysis by using nilpotent infinitesimals without any recourse to limits.
It is in the 19th century, in the midst of the Industrial Revolution in
Europe, that nilpotent infinitesimals were overtaken by so-called
$\varepsilon-\delta$ arguments. In the middle of the 20th century moribund
nilpotent infinitesimals were revived by Grothendieck in algebraic geometry
and by Lawvere in differential geometry. Differential geomery using nilpotent
infinitesimals consistently and coherently is now called \textit{synthetic
differential geometry}, since it prefers synthetic arguments to dull
calculations, as was the case in ancient Euclidean geometry. The central
object of study under synthetic differential geometry is microlinear spaces,
while the central object of study under orthodox differential geometry has
been finite-dimensional smooth manifolds. Since we can not expect to see
nilpotent infinitesimals in our real world, model theory of synthetic
differential geometry was developed vigorously by Dubuc and others around 1980
by using techniques of topos theory. For a standard reference on model theory
of synthetic differential geometry, the reader is referred to Kock's
well-written \cite{kock}.

It is well known that the category of smooth manifolds, whether
finite-dimensional or infinite-dimensional (modelled after Hilbert spaces,
Banach spaces, Fr\'{e}chet spaces or, most generally, convenient vector
spaces) is by no means cartesian closed, while the category of Fr\"{o}licher
spaces and smooth mappings among them is certainly so. Although nilpotent
infinitesimals are not visible in our real world, which has harassed such
philosophers of the 17th and 18th centuries as extremely skeptical Berkeley,
the notion of Weil functor is applicable to both finite-dimensional and
infinite-dimensional smooth manifolds. The notion of Weil functor has been
generalized in our previous paper \cite{nishi} so that it is applicable to
Fr\"{o}licher spaces. Therein we have delineated the class of Fr\"{o}licher
spaces which give credence to the illusion that Weil functors are the
exponentiation by infinitesimal objects in the shade. Such Fr\"{o}licher
spaces, which were shown to form a cartesian closed category, are called
\textit{Weil exponentiable}. In our subsequent paper \cite{nishi1} we have
finally succeeded in externalizing the notion of microlinearity, so that it is
applicable to Fr\"{o}licher spaces in the real world. Therein we have shown
that Fr\"{o}licher spaces that are Weil exponentiable and microlinear form a
cartesian closed category. We would like to hold that such Fr\"{o}licher
spaces are the central object of study for infinite-dimensional differential geometry.

The aim in our series of papers, the first of which is the present paper under
your eyes, is to develop differential geometry for Fr\"{o}licher spaces that
are Weil exponentiable and microlinear. The present paper is devoted to
tangent spaces and vector fields. Its principal result is that vector fields
on such a Fr\"{o}licher space form a Lie algebra. The reader should note that
our present mission for microlinear Fr\"{o}licher spaces is not so easy as
that for smooth manifolds in orthodox differential geometry, since coordinates
are no longer available nor are we admitted to identity vector fields with
their associated differential operators.

\section{Preliminaries}

\subsection{Fr\"{o}licher Spaces}

Fr\"{o}licher and his followers have vigorously and consistently developed a
general theory of smooth spaces, often called \textit{Fr\"{o}licher spaces}
for his celebrity, which were intended to be the \textit{underlying set
theory} for infinite-dimensional differential geometry in a sense. A
Fr\"{o}licher space is an underlying set endowed with a class of real-valued
functions on it (simply called \textit{structure} \textit{functions}) and a
class of mappings from the set $\mathbf{R}$ of real numbers to the underlying
set (called \textit{structure} \textit{curves}) subject to the condition that
structure curves and structure functions should compose so as to yield smooth
mappings from $\mathbf{R}$ to itself. It is required that the class of
structure functions and that of structure curves should determine each other
so that each of the two classes is maximal with respect to the other as far as
they abide by the above condition. What is most important among many nice
properties about the category $\mathbf{FS}$ of Fr\"{o}licher spaces and smooth
mappings is that it is cartesian closed, while neither the category of
finite-dimensional smooth manifolds nor that of infinite-dimensional smooth
manifolds modelled after any infinite-dimensional vector spaces such as
Hilbert spaces, Banach spaces, Fr\'{e}chet spaces or the like is so at all.
For a standard reference on Fr\"{o}licher spaces the reader is referred to
\cite{fro}.

\subsection{Weil Algebras and Infinitesimal Objects}

The notion of a \textit{Weil algebra} was introduced by Weil himself in
\cite{wei}. We denote by $\mathbf{W}$ the category of Weil algebras. Roughly
speaking, each Weil algebra corresponds to an infinitesimal object in the
shade. By way of example, the Weil algebra $\mathbf{R}[X]/(X^{2})$ (=the
quotient ring of the polynomial ring $\mathbf{R}[X]$\ of an indeterminate
$X$\ modulo the ideal $(X^{2})$\ generated by $X^{2}$) corresponds to the
infinitesimal object of first-order nilpotent infinitesimals, while the Weil
algebra $\mathbf{R}[X]/(X^{3})$ corresponds to the infinitesimal object of
second-order nilpotent infinitesimals. Although an infinitesimal object is
undoubtedly imaginary in the real world, as has harassed both mathematicians
and philosophers of the 17th and the 18th centuries because mathematicians at
that time preferred to talk infinitesimal objects as if they were real
entities, each Weil algebra yields its corresponding \textit{Weil functor} on
the category of smooth manifolds of some kind to itself, which is no doubt a
real entity. Intuitively speaking, the Weil functor corresponding to a Weil
algebra stands for the exponentiation by the infinitesimal object
corresponding to the Weil algebra at issue. For Weil functors on the category
of finite-dimensional smooth manifolds, the reader is referred to \S 35 of
\cite{kolar}, while the reader can find a readable treatment of Weil functors
on the category of smooth manifolds modelled on convenient vector spaces in
\S 31 of \cite{kri}.

\textit{Synthetic differential geometry }(usually abbreviated to SDG), which
is a kind of differential geometry with a cornucopia of nilpotent
infinitesimals, was forced to invent its models, in which nilpotent
infinitesimals were visible. For a standard textbook on SDG, the reader is
referred to \cite{lav}, while he or she is referred to \cite{kock} for the
model theory of SDG vigorously constructed by Dubuc \cite{dub} and others.
Although we do not get involved in SDG herein, we will exploit locutions in
terms of infinitesimal objects so as to make the paper highly readable. Thus
we prefer to write $W_{D}$\ and $W_{D_{2}}$\ in place of $\mathbf{R}%
[X]/(X^{2})$ and $\mathbf{R}[X]/(X^{3})$ respectively, where $D$ stands for
the infinitesimal object of first-order nilpotent infinitesimals, and $D_{2}%
$\ stands for the infinitesimal object of second-order nilpotent
infinitesimals. To Newton and Leibniz, $D$ stood for
\[
\{d\in\mathbf{R}\mid d^{2}=0\}
\]
while $D_{2}$\ stood for
\[
\{d\in\mathbf{R}\mid d^{3}=0\}
\]
We will write $W_{d\in D_{2}\mapsto d^{2}\in D}$ for the homomorphim of Weil
algebras $\mathbf{R}[X]/(X^{2})\rightarrow\mathbf{R}[X]/(X^{3})$ induced by
the homomorphism $X\rightarrow X^{2}$ of the polynomial ring \ $\mathbf{R}[X]
$ to itself. Such locutions are justifiable, because the category $\mathbf{W}
$ of Weil algebras in the real world and the category of infinitesimal objects
in the shade are dual to each other in a sense. To familiarize himself or
herself with such locutions, the reader is strongly encouraged to read the
first two chapters of \cite{lav}, even if he or she is not interested in SDG
at all.

We need to fix notation and terminology for simplicial objects, which form an
important subclass of infinitesimal objects. \textit{Simplicial objects} are
infinitesimal objects of the form
\begin{align*}
&  D^{n}\{\mathfrak{p}\}\\
&  =\{(d_{1},...,d_{n})\in D^{n}\mid d_{i_{1}}...d_{i_{k}}=0\text{ }%
(\forall(i_{1},...,i_{k})\in\mathfrak{p)\}}%
\end{align*}
where $\mathfrak{p}$\ is a finite set of finite sequences $(i_{1},...,i_{k}%
)$\ of natural numbers between $1$ and $n$, including the endpoints, with
$i_{1}<...<i_{k}$. If $\mathfrak{p}$\ is empty, $D^{n}\{\mathfrak{p}\}$ is
$D^{n}$ itself. If $\mathfrak{p}$ consists of all the binary sequences, then
$D^{n}\{\mathfrak{p}\}$ represents $D(n)$ in the standard terminology of SDG.
Given two simplicial objects $D^{m}\{\mathfrak{p}\}$\ and $D^{n}%
\{\mathfrak{q}\}$, we define a simplicial object $D^{m}\{\mathfrak{p}\}\oplus
D^{n}\{\mathfrak{q}\}$ to be
\[
D^{m+n}\{\mathfrak{p}\oplus\mathfrak{q}\}
\]
where
\begin{align*}
&  \mathfrak{p}\oplus\mathfrak{q}\\
&  =\mathfrak{p}\cup\{(j_{1}+m,...,j_{k}+m)\mid(j_{1},...,j_{k})\in
\mathfrak{q}\}\\
\cup\{(i,j+m)  &  \mid1\leq i\leq m,\ 1\leq j\leq n\}
\end{align*}
Since the operation $\oplus$\ is associative, we can combine any finite number
of simplicial objects by $\oplus$ without bothering about how to insert
parentheses. Given morphisms of simplicial objects $\Phi_{i}:D^{m_{i}%
}\{\mathfrak{p}_{i}\}\rightarrow D^{m}\{\mathfrak{p}\}\ (1\leq i\leq n)$,
there exists a unique morphism of simplicial objects $\Phi:D^{m_{1}%
}\{\mathfrak{p}_{1}\}\oplus...\oplus D^{m_{n}}\{\mathfrak{p}_{n}\}\rightarrow
D^{m}\{\mathfrak{p}\}$ whose restriction to $D^{m_{i}}\{\mathfrak{p}_{i}\}$
coincides with $\Phi_{i}$ for each $i$. We denote this $\Phi$\ by $\Phi
_{1}\oplus...\oplus\Phi_{n}$. We write $D(n)$ for $\{(d,...,d)\in D^{n}\mid
d_{i}d_{j}=0$ for any $i\neq j\}$.

\subsection{Microlinearity}

In \cite{nishi} we have discussed how to assign, to each pair $(X,W)$\ of a
Fr\"{o}licher space $X$ and a Weil algebra $W$,\ another Fr\"{o}licher space
$X\otimes W$\ called the \textit{Weil prolongation of} $X$ \textit{with
respect to} $W$, which is naturally extended to a bifunctor $\mathbf{FS}%
\times\mathbf{W\rightarrow FS}$, and then to show that the functor
$\cdot\otimes W:\mathbf{FS\rightarrow FS}$ is product-preserving for any Weil
algebra $W$. Weil prolongations are well-known as \textit{Weil functors} for
finite-dimensional and infinite-dimensional smooth manifolds in orthodox
differential geometry, as we have already touched upon in the preceding subsection.

The central object of study in SDG is \textit{microlinear} spaces. Although
the notion of a manifold (=a pasting of copies of a certain linear space) is
defined on the local level, the notion of microlinearity is defined absolutely
on the genuinely infinitesimal level. For the historical account of
microlinearity, the reader is referred to \S \S 2.4 of \cite{lav} or Appendix
D of \cite{kock}. To get an adequately restricted cartesian closed subcategory
of Fr\"{o}licher spaces, we have emancipated microlinearity from within a
well-adapted model of SDG to Fr\"{o}licher spaces in the real world in
\cite{nishi1}. Recall that a Fr\"{o}licher space $X$ is called
\textit{microlinear} providing that any finite limit diagram $\mathcal{D}$ in
$\mathbf{W}$ yields a limit diagram $X\otimes\mathcal{D}$ in $\mathbf{FS}$,
where $X\otimes\mathcal{D}$ is obtained from $\mathcal{D}$ by putting
$X\otimes$ to the left of every object and every morphism in $\mathcal{D}$. As
we have discussed there, all convenient vector spaces are microlinear, so that
all $C^{\infty}$-manifolds in the sense of \cite{kri} (cf. Section 27) are
also microlinear.

We have no reason to hold that all Fr\"{o}licher spaces credit Weil
prolongations as exponentiation by infinitesimal objects in the shade.
Therefore we need a notion which distinguishes Fr\"{o}licher spaces that do so
from those that do not. A Fr\"{o}licher space $X$ is called \textit{Weil
exponentiable }if
\begin{equation}
(X\otimes(W_{1}\otimes_{\infty}W_{2}))^{Y}=(X\otimes W_{1})^{Y}\otimes
W_{2}\label{2.1}%
\end{equation}
holds naturally for any Fr\"{o}licher space $Y$ and any Weil algebras $W_{1}$
and $W_{2}$. If $Y=1$, then (\ref{2.1}) degenerates into
\begin{equation}
X\otimes(W_{1}\otimes_{\infty}W_{2})=(X\otimes W_{1})\otimes W_{2}\label{2.2}%
\end{equation}
If $W_{1}=\mathbf{R}$, then (\ref{2.1}) degenerates into
\begin{equation}
(X\otimes W_{2})^{Y}=X^{Y}\otimes W_{2}\label{2.3}%
\end{equation}
We have shown in \cite{nishi} that all convenient vector spaces are Weil
exponentiable, so that all $C^{\infty}$-manifolds in the sense of \cite{kri}
(cf. Section 27) are Weil exponentiable.

We have demonstrated in \cite{nishi1} that all Fr\"{o}licher spaces that are
microlinear and Weil exponentiable form a cartesian closed category. In the
sequel $M$ is assumed to be such a Fr\"{o}licher space.

\subsection{Comma Categories}

In the next section we will work on the category $\mathbf{FS}/M$ of
$\mathbf{FS}$ over $M$, for which we need to know that the product of
$N_{1}\rightarrow M$ and $N_{1}\rightarrow M$ in $\mathbf{FS}/M$ is no other
than the fibred product $N_{1}\times_{M}N_{2}\rightarrow M$. For the general
theory of comma categories, the reader is referred to \cite{irf1}, \cite{irf2}
and \cite{irf3}.

\section{\label{s3}Tangent Bundles}

\begin{proposition}
\label{t3.1}If $M$ is a microlinear Fr\"{o}licher space, then we have
\[
M\otimes W_{D(2)}=(M\otimes W_{D})\times_{M}(M\otimes W_{D})
\]

\end{proposition}

\begin{proof}
We have the following pullback diagram of Weil algebras:
\begin{equation}%
\begin{array}
[c]{ccc}%
W_{D(2)} & \rightarrow & W_{D}\\
\downarrow &  & \downarrow\\
W_{D} & \rightarrow & W_{1}%
\end{array}
\label{3.1.1}%
\end{equation}
where the left vertical arrow is induced by the mapping
\[
d\in D\mapsto(d,0)\in D(2)
\]
while the upper horizontal arrow is induced by the mapping
\[
d\in D\mapsto(0,d)\in D(2)
\]
The above pullback diagram naturally gives rise to the following pullback
diagram because of the microlinearity of $M$:
\[%
\begin{array}
[c]{ccc}%
M\otimes W_{D(2)} & \rightarrow & M\otimes W_{D}\\
\downarrow &  & \downarrow\\
M\otimes W_{D} & \rightarrow & M\otimes W_{1}=M
\end{array}
\]
This completes the proof.
\end{proof}

\begin{corollary}
The canonical projection $M\otimes W_{D(2)}\rightarrow M$ is a biproduct of
two copies of $M\otimes W_{D}\rightarrow M$ in the comma category
$\mathbf{FS}/M$.
\end{corollary}

Now we are in a position to define basic operations on $M\otimes
W_{D}\rightarrow M$ in the category $\mathbf{FS}/M$ so as to make it a vector
space over the projection $\mathbb{R\times}M\rightarrow M$.

\begin{enumerate}
\item The addition is defined by
\[
\mathrm{id}_{M}\otimes W_{+_{D}}:M\otimes W_{D(2)}\rightarrow M\otimes W_{D}%
\]
where the putative mapping $+_{D}:D\rightarrow D(2)$ is
\[
+_{D}:d\in D\mapsto(d,d)\in D(2)
\]

\item The identity with respect to the above addition is defined by
\[
\mathrm{id}_{M}\otimes W_{\mathbf{0}_{D}}:M=M\otimes W_{1}\rightarrow M\otimes
W_{D}%
\]
where the putative mapping $\mathbf{0}_{D}:D\rightarrow1$ is the unique mapping.

\item The inverse with respect to the above addition is defined by
\[
\mathrm{id}_{M}\otimes W_{-_{D}}:M\otimes W_{D}\rightarrow M\otimes W_{D}%
\]
where the putative mapping $-_{D}:D\rightarrow D$ is
\[
-_{D}:d\in D\mapsto-d\in D
\]

\item The scalar multiplication by a scalar $\alpha\in\mathbf{R}$ is defined
by
\[
\mathrm{id}_{M}\otimes W_{\alpha_{D}}:M\otimes W_{D}\rightarrow M\otimes W_{D}%
\]
where the putative mapping $\alpha_{D}:D\rightarrow D$ is
\[
\alpha_{D}:d\in D\mapsto\alpha d\in D
\]

\end{enumerate}

\begin{theorem}
\label{t3.2}The canonical projection $M\otimes W_{D}\rightarrow M$ is a vector
space over the projection $\mathbf{R}\mathbb{\times}M\rightarrow M$\ in the
category $\mathbf{FS}/M$.
\end{theorem}

\begin{proof}
\begin{enumerate}
\item The associativity of the addition follows from the following commutative
diagram:
\[%
\begin{array}
[c]{ccccc}
&  & \mathrm{id}_{M}\otimes W_{\epsilon_{23}} &  & \\
& M\otimes W_{D(3)} & \rightarrow & M\otimes W_{D(2)} & \\
\mathrm{id}_{M}\otimes W_{\epsilon_{12}} & \downarrow &  & \downarrow &
\mathrm{id}_{M}\otimes W_{+_{D}}\\
& M\otimes W_{D(2)} & \rightarrow & M\otimes W_{D} & \\
&  & \mathrm{id}_{M}\otimes W_{+_{D}} &  &
\end{array}
\]
where the putative mapping $\epsilon_{23}:D(2)\rightarrow D(3)$ is
\[
(d_{1},d_{2})\in D(2)\mapsto(d_{1},d_{1},d_{2})\in D(3)
\]
while the putative mapping $\epsilon_{12}:D(2)\rightarrow D(3)$ is
\[
(d_{1},d_{2})\in D(2)\mapsto(d_{1},d_{2},d_{2})\in D(3)
\]

\item The commutativity of the addition follows readily from the commutative
diagram
\[%
\begin{array}
[c]{cccc}
& W_{+_{D}} &  & \\
W_{D} & \leftarrow & W_{D(2)} & \\
& \nwarrow & \uparrow & W_{\tau}\\
W_{+_{D}} &  & W_{D(2)} &
\end{array}
\]
where the putative mapping $\tau:D(2)\rightarrow D(2)$ is
\[
(d_{1},d_{2})\in D(2)\mapsto(d_{2},d_{1})\in D(2)\text{.}%
\]

\item To see that the identity defined above really plays the identity with
respect to the above addition, it suffices to note that the composition of the
following two putative mappings
\[
d\in D\mapsto(d,0)\in D(2)
\]
\[
(d_{1},d_{2})\in D(2)\mapsto d_{1}\in D
\]
in order is the identity mapping of $D$, while the composition of the
following two putative mappings
\[
d\in D\mapsto(0,d)\in D(2)
\]
\[
(d_{1},d_{2})\in D(2)\mapsto d_{1}\in D
\]
in order is the constant mapping
\[
d\in D\mapsto0\in D
\]

\item To see that the addition of scalars distributes with respect to the
scalar multiplication, it suffices to note that, for any $\alpha_{1}%
,\alpha_{2}\in\mathbf{R}$, the composition of the following two putative
mappings
\[
d\in D\mapsto(d,0)\in D(2)
\]
\[
(d_{1},d_{2})\in D(2)\mapsto\alpha_{1}d_{1}+\alpha_{2}d_{2}\in D
\]
in order is the mapping
\[
d\in D\mapsto\alpha_{1}d_{1}\in D
\]
and the composition of the two putative mappings
\[
d\in D\mapsto(0,d)\in D(2)
\]
\[
(d_{1},d_{2})\in D(2)\mapsto\alpha_{1}d_{1}+\alpha_{2}d_{2}\in D
\]
in order is the mapping
\[
d\in D\mapsto\alpha_{2}d_{2}\in D
\]
while the composition of the two putative mappings
\[
d\in D\mapsto(d,d)\in D(2)
\]
\[
(d_{1},d_{2})\in D(2)\mapsto\alpha_{1}d_{1}+\alpha_{2}d_{2}\in D
\]
in order is no other than the mapping
\[
d\in D\mapsto(\alpha_{1}+\alpha_{2})d\in D
\]

\item To see that the addition of vectors distributes with respect to the
scalar multiplication, it suffices to note that, for any $\alpha\in\mathbf{R}
$, the composition of the two putative mappings
\[
d\in D\mapsto(d,0)\in D(2)
\]
\[
(d_{1},d_{2})\in D(2)\mapsto(\alpha d_{1},\alpha d_{2})\in D(2)
\]
in order is the composition of the two putative mappings
\[
d\in D\mapsto\alpha d\in D
\]
\[
d\in D\mapsto(d,0)\in D(2)
\]
in order, and the composition of the two putative mappings
\[
d\in D\mapsto(0,d)\in D(2)
\]
\[
(d_{1},d_{2})\in D(2)\mapsto(\alpha d_{1},\alpha d_{2})\in D(2)
\]
in order is the composition of the two putative mappings
\[
d\in D\mapsto\alpha d\in D
\]
\[
d\in D\mapsto(0,d)\in D(2)
\]
in order, while the composition of the two putative mappings
\[
d\in D\mapsto(d,d)\in D(2)
\]
\[
(d_{1},d_{2})\in D(2)\mapsto(\alpha d_{1},\alpha d_{2})\in D(2)
\]
in order is no other than the composition of the two putative mappings
\[
d\in D\mapsto\alpha d\in D
\]
\[
d\in D\mapsto(d,d)\in D(2)
\]
in order.

\item The verification of the other axioms for a vector space can safely be
left to the reader.
\end{enumerate}
\end{proof}

\begin{corollary}
For any $x\in M$, the inverse image $(M\otimes W_{D})_{x}$\ of $x$ under the
canonical projection $M\otimes W_{D}\rightarrow M$ is a vector space over
$\mathbf{R}$.
\end{corollary}

The proof of the following easy proposition is left to the reader.

\begin{proposition}
For any $t_{1},...,t_{n}\in(M\otimes W_{D})_{x}$, there exists a unique
$l_{(t_{1},...,t_{n})}\in M\otimes W_{D(n)}$ with
\[
\mathrm{id}_{M}\otimes W_{i_{j}^{D(n)}}(l_{(t_{1},...,t_{n})})=t_{j}%
\qquad(j=1,...,n)
\]
where the putative mapping $i_{j}^{D(n)}:D\rightarrow D(n)$ is
\[
d\in D\mapsto(0,...,0,\underset{i}{d},0,...,0)\in D(n)
\]

\end{proposition}

\section{Vector Fields}

Vector fields can be delineated in two distinct ways.

\begin{theorem}
\label{t4.1}The space of sections of the tangent bundle $M\otimes
W_{D}\rightarrow M$ can naturally be identified with the space $(M^{M}\otimes
W_{D})_{\mathrm{id}_{M}}$.
\end{theorem}

\begin{proof}
This follows simply from the following instance of the Weil exponentiability
of $M$:
\[
M^{M}\otimes W_{D}=(M\otimes W_{D})^{M}%
\]

\end{proof}

\begin{remark}
In this paper the viewpoint of a vector field on $M$ as an element of
$(M^{M}\otimes W_{D})_{\mathrm{id}_{M}}$\ is preferred to that as a section of
the tangent bundle $M\otimes W_{D}\rightarrow M$.
\end{remark}

\begin{notation}
The totality of vector fields on $M$ is denoted by $\aleph(M)$.
\end{notation}

\begin{definition}
For any $\gamma_{1}\in M^{M}\otimes W_{D^{m}}$ and any $\gamma_{2}\in
M^{M}\otimes W_{D^{n}}$ we define $\gamma_{2}\ast\gamma_{1}\in M^{M}\otimes
W_{D^{m+n}}$ to be
\[
(\circ_{M^{M}}\otimes\mathrm{id}_{W_{D^{m+n}}})((\mathrm{id}_{M^{M}}\otimes
W_{p_{D^{m}}^{D^{m+n}}})(\gamma_{1}),(\mathrm{id}_{M^{M}}\otimes W_{p_{D^{n}%
}^{D^{m+n}}})(\gamma_{2}))
\]
where $\circ_{M^{M}}$ is the bifunctor assigning the composition $g\circ f$ to
each pair $(f,g)\in M^{M}\times M^{M}$ while $p_{D^{m}}^{D^{m+n}}$%
:$D^{m+n}\rightarrow D^{m}$ and $p_{D^{n}}^{D^{m+n}}$:$D^{m+n}\rightarrow
D^{n} $ are the canonical projections.
\end{definition}

\begin{proposition}
\label{t4.2}For any $\gamma_{1}\in M^{M}\otimes W_{D^{l}}$, any $\gamma_{2}\in
M^{M}\otimes W_{D^{m}}$ and any $\gamma_{3}\in M^{M}\otimes W_{D^{n}}$, we
have
\[
\gamma_{3}\ast(\gamma_{2}\ast\gamma_{1})=(\gamma_{3}\ast\gamma_{2})\ast
\gamma_{1}%
\]
In other words, the operation $\ast$ is associative.
\end{proposition}

\begin{proof}
By dint of the bifunctionality of $\otimes$,\ it is easy to see that the
diagram
\begin{equation}%
\begin{array}
[c]{ccc}
& \circ_{M^{M}}\otimes\mathrm{id}_{W_{D^{l+m}}} & \\
(M^{M}\otimes W_{D^{l+m}})\times(M^{M}\otimes W_{D^{l+m}}) & \rightarrow &
M^{M}\otimes W_{D^{l+m}}\\
\downarrow &  & \downarrow\\
(M^{M}\otimes W_{D^{l+m+n}})\times(M^{M}\otimes W_{D^{l+m+n}}) & \rightarrow &
M^{M}\otimes W_{D^{l+m+n}}\\
& \circ_{M^{M}}\otimes\mathrm{id}_{W_{D^{l+m+n}}} &
\end{array}
\label{4.2.1}%
\end{equation}
is commutative, where the left vertical arrow stands for $(\mathrm{id}_{M^{M}%
}\otimes W_{p_{D^{l+m}}^{D^{l+m+n}}})\times(\mathrm{id}_{M^{M}}\otimes
W_{p_{D^{l+m}}^{D^{l+m+n}}})$, and the right vertical arrow stands for
$\mathrm{id}_{M^{M}}\otimes W_{p_{D^{l+m}}^{D^{l+m+n}}}$. By dint of the
bifunctionality of $\otimes$ again,\ it is also easy to see that the diagram
\begin{equation}%
\begin{array}
[c]{ccc}
& \circ_{M^{M}}\otimes\mathrm{id}_{W_{D^{m+n}}} & \\
(M^{M}\otimes W_{D^{m+n}})\times(M^{M}\otimes W_{D^{m+n}}) & \rightarrow &
M^{M}\otimes W_{D^{m+n}}\\
\downarrow &  & \downarrow\\
(M^{M}\otimes W_{D^{l+m+n}})\times(M^{M}\otimes W_{D^{l+m+n}}) & \rightarrow &
M^{M}\otimes W_{D^{l+m+n}}\\
& \circ_{M^{M}}\otimes\mathrm{id}_{W_{D^{l+m+n}}} &
\end{array}
\label{4.2.2}%
\end{equation}
is commutative, where the left vertical arrow stands for $(\mathrm{id}_{M^{M}%
}\otimes W_{p_{D^{m+n}}^{D^{l+m+n}}})\times(\mathrm{id}_{M^{M}}\otimes
W_{p_{D^{m+n}}^{D^{l+m+n}}})$, and the right vertical arrow stands for
$\mathrm{id}_{M^{M}}\otimes W_{p_{D^{m+n}}^{D^{l+m+n}}}$. Therefore we have
\begin{align*}
&  \gamma_{3}\ast(\gamma_{2}\ast\gamma_{1})\\
&  =(\circ_{M^{M}}\otimes\mathrm{id}_{W_{D^{l+m+n}}})((\mathrm{id}_{M^{M}%
}\otimes W_{p_{Dl+m}^{D^{l+m+n}}})((\circ_{M^{M}}\otimes\mathrm{id}%
_{W_{D^{l+m}}})\\
&  ((\mathrm{id}_{M^{M}}\otimes W_{p_{D^{l}}^{D^{l+m}}})(\gamma_{1}%
),(\mathrm{id}_{M^{M}}\otimes W_{p_{D^{m}}^{D^{l+m}}})(\gamma_{2}%
))),(\mathrm{id}_{M^{M}}\otimes W_{p_{D^{n}}^{D^{l+m+n}}})(\gamma_{3}))\\
&  =(\circ_{M^{M}}\otimes\mathrm{id}_{W_{D^{l+m+n}}})((\circ_{M^{M}}%
\otimes\mathrm{id}_{W_{D^{l+m+n}}})((\mathrm{id}_{M^{M}}\otimes W_{p_{Dl+m}%
^{D^{l+m+n}}})\circ(\mathrm{id}_{M^{M}}\otimes W_{p_{D^{l}}^{D^{l+m}}}%
)(\gamma_{1}),\\
&  (\mathrm{id}_{M^{M}}\otimes W_{p_{Dl+m}^{D^{l+m+n}}})\circ(\mathrm{id}%
_{M^{M}}\otimes W_{p_{D^{m}}^{D^{l+m}}})(\gamma_{2})),(\mathrm{id}_{M^{M}%
}\otimes W_{p_{D^{n}}^{D^{l+m+n}}})(\gamma_{3}))\\
&  =(\circ_{M^{M}}\otimes\mathrm{id}_{W_{D^{l+m+n}}})((\circ_{M^{M}}%
\otimes\mathrm{id}_{W_{D^{l+m+n}}})((\mathrm{id}_{M^{M}}\otimes W_{p_{D^{l}%
}^{D^{l+m+n}}})(\gamma_{1}),\\
&  (\mathrm{id}_{M^{M}}\otimes W_{p_{D^{m}}^{D^{l+m+n}}})(\gamma
_{2})),(\mathrm{id}_{M^{M}}\otimes W_{p_{D^{n}}^{D^{l+m+n}}})(\gamma_{3}))\\
&  \text{[by dint of the commutativity of the diagram (\ref{4.2.1})]}\\
&  =(\circ_{M^{M}}\otimes\mathrm{id}_{W_{D^{l+m+n}}})((\mathrm{id}_{M^{M}%
}\otimes W_{p_{D^{l}}^{D^{l+m+n}}})(\gamma_{1}),(\circ_{M^{M}}\otimes
\mathrm{id}_{W_{D^{l+m+n}}})((\mathrm{id}_{M^{M}}\otimes W_{p_{D^{m}%
}^{D^{l+m+n}}})(\gamma_{2}),\\
&  (\mathrm{id}_{M^{M}}\otimes W_{p_{D^{n}}^{D^{l+m+n}}})(\gamma_{3})))\\
&  \text{[since the operation }\circ_{M^{M}}\text{ is associative]}\\
&  =(\circ_{M^{M}}\otimes\mathrm{id}_{W_{D^{l+m+n}}})((\mathrm{id}_{M^{M}%
}\otimes W_{p_{D^{l}}^{D^{l+m+n}}})(\gamma_{1}),(\circ_{M^{M}}\otimes
\mathrm{id}_{W_{D^{l+m+n}}})((\mathrm{id}_{M^{M}}\otimes W_{p_{D^{m+n}%
}^{D^{l+m+n}}})\circ\\
&  (\mathrm{id}_{M^{M}}\otimes W_{p_{D^{m}}^{D^{m+n}}})(\gamma_{2}%
),(\mathrm{id}_{M^{M}}\otimes W_{p_{D^{m+n}}^{D^{l+m+n}}})\circ(\mathrm{id}%
_{M^{M}}\otimes W_{p_{D^{n}}^{D^{m+n}}})(\gamma_{3})))\\
&  \text{[by dint of the commutativity of the diagram (\ref{4.2.2})]}\\
&  =(\circ_{M^{M}}\otimes\mathrm{id}_{W_{D^{l+m+n}}})((\mathrm{id}_{M^{M}%
}\otimes W_{p_{D^{l}}^{D^{l+m+n}}})(\gamma_{1}),(\mathrm{id}_{M^{M}}\otimes
W_{p_{D^{m+n}}^{D^{l+m+n}}})((\circ_{M^{M}}\otimes\mathrm{id}_{W_{D^{m+n}}})\\
&  ((\mathrm{id}_{M^{M}}\otimes W_{p_{D^{m}}^{D^{m+n}}})(\gamma_{2}%
),(\mathrm{id}_{M^{M}}\otimes W_{p_{D^{n}}^{D^{m+n}}})(\gamma_{3}))))\\
&  =(\gamma_{3}\ast\gamma_{2})\ast\gamma_{1}%
\end{align*}

\end{proof}

It is easy to see that

\begin{proposition}
\label{t4.3}For any $\gamma\in M^{M}\otimes W_{D^{m}}$ we have
\begin{align*}
\gamma\ast I_{l}  &  =(\mathrm{id}_{M^{M}}\otimes W_{p_{D^{m}}^{D^{l+m}}%
})(\gamma)\\
I_{n}\ast\gamma &  =(\mathrm{id}_{M^{M}}\otimes W_{p_{D^{m}}^{D^{m+n}}%
})(\gamma)
\end{align*}
where $I_{l}:M^{M}\otimes W_{D^{l}}\rightarrow M^{M}\otimes W_{D^{l}}$ and
$I_{n}:M^{M}\otimes W_{D^{n}}\rightarrow M^{M}\otimes W_{D^{n}}$ are the
identity mappings.
\end{proposition}

The following proposition is essentially a variant of Proposition 3 in
\S \S 3.2 of \cite{lav}.

\begin{proposition}
\label{t4.4}For any $X\in\aleph(M)$, we have
\begin{align*}
&  (\mathrm{id}_{M^{M}}\otimes W_{(d_{1},d_{2})\in D(2)\mapsto d_{1}+d_{2}\in
D})(X)\\
&  =(\circ_{M^{M}}\otimes\mathrm{id}_{W_{D(2)}})((\mathrm{id}_{M^{M}}\otimes
W_{p_{1}^{D(2)}})(X),(\mathrm{id}_{M^{M}}\otimes W_{p_{2}^{D(2)}})(X))
\end{align*}
where $p_{1}^{D(2)}:D(2)\rightarrow D$ and $p_{2}^{D(2)}:D(2)\rightarrow D$
are the canonical projections onto the first and second factors.
\end{proposition}

\begin{proof}
With due regard to the limit diagram (\ref{3.1.1}) of Weil algebras and the
microlinearity of $M^{M}$, it suffices to see that
\begin{align}
&  (\mathrm{id}_{M^{M}}\otimes W_{i_{1}^{D(2)}})((\mathrm{id}_{M^{M}}\otimes
W_{(d_{1},d_{2})\in D(2)\mapsto d_{1}+d_{2}\in D})(X))\nonumber\\
&  =(\mathrm{id}_{M^{M}}\otimes W_{i_{1}^{D(2)}})((\circ_{M^{M}}%
\otimes\mathrm{id}_{W_{D(2)}})((\mathrm{id}_{M^{M}}\otimes W_{p_{1}^{D(2)}%
})(X),(\mathrm{id}_{M^{M}}\otimes W_{p_{2}^{D(2)}})(X)))\label{4.4.1}%
\end{align}
and that
\begin{align}
&  (\mathrm{id}_{M^{M}}\otimes W_{i_{2}^{D(2)}})((\mathrm{id}_{M^{M}}\otimes
W_{(d_{1},d_{2})\in D(2)\mapsto d_{1}+d_{2}\in D})(X))\nonumber\\
&  =(\mathrm{id}_{M^{M}}\otimes W_{i_{2}^{D(2)}})((\circ_{M^{M}}%
\otimes\mathrm{id}_{W_{D(2)}})((\mathrm{id}_{M^{M}}\otimes W_{p_{1}^{D(2)}%
})(X),(\mathrm{id}_{M^{M}}\otimes W_{p_{2}^{D(2)}})(X)))\label{4.4.2}%
\end{align}
Here we deal only with (\ref{4.4.1}), leaving the similar treatment of
(\ref{4.4.2}) safely to the reader. By dint of the bifunctionality of
$\otimes$,\ it is easy to see that the diagram
\begin{equation}%
\begin{array}
[c]{ccc}
& \circ_{M^{M}}\otimes\mathrm{id}_{W_{D(2)}} & \\
(M^{M}\otimes W_{D(2)})\times(M^{M}\otimes W_{D(2)}) & \rightarrow &
M^{M}\otimes W_{D(2)}\\
\downarrow &  & \downarrow\\
(M^{M}\otimes W_{D})\times(M^{M}\otimes W_{D}) & \rightarrow & M^{M}\otimes
W_{D}\\
& \circ_{M^{M}}\otimes\mathrm{id}_{W_{D}} &
\end{array}
\label{4.4.3}%
\end{equation}
commutes, where the left vertical arrow stands for $(\mathrm{id}_{M^{M}%
}\otimes W_{i_{1}^{D(2)}})\times(\mathrm{id}_{M^{M}}\otimes W_{i_{1}^{D(2)}}%
)$, and the right vertical arrow stands for $\mathrm{id}_{M^{M}}\otimes
W_{i_{1}^{D(2)}}$ with $i_{1}^{D(2)}:D\rightarrow D(2)$ being the canonical
injection $d\in D\mapsto(d,0)\in D(2)$. Therefore we have
\begin{align*}
&  (\mathrm{id}_{M^{M}}\otimes W_{i_{1}^{D(2)}})((\circ_{M^{M}}\otimes
\mathrm{id}_{W_{D(2)}})((\mathrm{id}_{M^{M}}\otimes W_{p_{1}^{D(2)}%
})(X),(\mathrm{id}_{M^{M}}\otimes W_{p_{2}^{D(2)}})(X)))\\
&  =(\circ_{M^{M}}\otimes\mathrm{id}_{W_{D}})((\mathrm{id}_{M^{M}}\otimes
W_{i_{1}^{D(2)}})\circ(\mathrm{id}_{M^{M}}\otimes W_{p_{1}^{D(2)}})(X),\\
&  (\mathrm{id}_{M^{M}}\otimes W_{i_{1}^{D(2)}})\circ(\mathrm{id}_{M^{M}%
}\otimes W_{p_{2}^{D(2)}})(X)))\\
&  =(\circ_{M^{M}}\otimes\mathrm{id}_{W_{D}})((\mathrm{id}_{M^{M}}%
\otimes(W_{i_{1}^{D(2)}}\circ W_{p_{1}^{D(2)}}))(X),(\mathrm{id}_{M^{M}%
}\otimes(W_{i_{1}^{D(2)}}\circ W_{p_{2}^{D(2)}})(X))\\
&  =X
\end{align*}
while it is trivial to see that
\begin{align*}
&  (\mathrm{id}_{M^{M}}\otimes W_{i_{1}^{D(2)}})((\mathrm{id}_{M^{M}}\otimes
W_{(d_{1},d_{2})\in D(2)\mapsto d_{1}+d_{2}\in D})(X))\\
&  =(\mathrm{id}_{M^{M}}\otimes(W_{i_{1}^{D(2)}}\circ W_{(d_{1},d_{2})\in
D(2)\mapsto d_{1}+d_{2}\in D}))(X)\\
&  =X
\end{align*}
Thus we are done.
\end{proof}

The following corollary is essentially a variant of Proposition 4 in \S \S 3.2
of \cite{lav}.

\begin{corollary}%
\begin{align*}
&  (\mathrm{id}_{M^{M}}\otimes W_{d\in D\mapsto(d,-d)\in D(2)})((\circ_{M^{M}%
}\otimes\mathrm{id}_{W_{D(2)}})((\mathrm{id}_{M^{M}}\otimes W_{p_{1}^{D(2)}%
})(X),\\
&  (\mathrm{id}_{M^{M}}\otimes W_{p_{2}^{D(2)}})(X)))\\
&  =(\mathrm{id}_{M^{M}}\otimes W_{d\in D\mapsto(-d,d)\in D(2)})((\circ
_{M^{M}}\otimes\mathrm{id}_{W_{D(2)}})((\mathrm{id}_{M^{M}}\otimes
W_{p_{1}^{D(2)}})(X),\\
&  (\mathrm{id}_{M^{M}}\otimes W_{p_{2}^{D(2)}})(X)))\\
&  =\mathbf{0}_{\mathrm{id}_{M}}%
\end{align*}

\end{corollary}

\begin{proof}
We have
\begin{align*}
&  (\mathrm{id}_{M^{M}}\otimes W_{d\in D\mapsto(d,-d)\in D(2)})((\circ_{M^{M}%
}\otimes\mathrm{id}_{W_{D(2)}})((\mathrm{id}_{M^{M}}\otimes W_{p_{1}^{D(2)}%
})(X),\\
&  (\mathrm{id}_{M^{M}}\otimes W_{p_{2}^{D(2)}})(X)))\\
&  =(\mathrm{id}_{M^{M}}\otimes W_{d\in D\mapsto(d,-d)\in D(2)})((\mathrm{id}%
_{M^{M}}\otimes W_{(d_{1},d_{2})\in D(2)\mapsto d_{1}+d_{2}\in D})(X))\\
&  \text{\lbrack by Proposition \ref{t4.4}]}\\
&  =(\mathrm{id}_{M^{M}}\otimes W_{d\in D\mapsto0\in D})(X)\\
&  =\mathbf{0}_{\mathrm{id}_{M}}%
\end{align*}
Similarly we have
\begin{align*}
&  (\mathrm{id}_{M^{M}}\otimes W_{d\in D\mapsto(-d,d)\in D(2)})((\circ_{M^{M}%
}\otimes\mathrm{id}_{W_{D(2)}})((\mathrm{id}_{M^{M}}\otimes W_{p_{1}^{D(2)}%
})(X),\\
&  (\mathrm{id}_{M^{M}}\otimes W_{p_{2}^{D(2)}})(X)))\\
&  =\mathbf{0}_{\mathrm{id}_{M}}%
\end{align*}

\end{proof}

The following proposition is essentially a variant of Proposition 6 in
\S \S 3.2 of \cite{lav}.

\begin{proposition}
\label{t4.5}For any $X,Y\in\aleph(M)$ we have
\begin{align*}
&  l_{(X,Y)}\\
&  =(\circ_{M^{M}}\otimes\mathrm{id}_{W_{D(2)}})((\mathrm{id}_{M^{M}}\otimes
W_{p_{1}^{D(2)}})(X),(\mathrm{id}_{M^{M}}\otimes W_{p_{2}^{D(2)}})(Y))\\
&  =(\circ_{M^{M}}\otimes\mathrm{id}_{W_{D(2)}})((\mathrm{id}_{M^{M}}\otimes
W_{p_{1}^{D(2)}})(Y),(\mathrm{id}_{M^{M}}\otimes W_{p_{2}^{D(2)}})(X))
\end{align*}
so that
\begin{align*}
&  X+Y\\
&  =(\mathrm{id}_{M^{M}}\otimes W_{+_{D}})((\circ_{M^{M}}\otimes
\mathrm{id}_{W_{D(2)}})((\mathrm{id}_{M^{M}}\otimes W_{p_{1}^{D(2)}%
})(X),(\mathrm{id}_{M^{M}}\otimes W_{p_{2}^{D(2)}})(Y)))\\
&  =(\mathrm{id}_{M^{M}}\otimes W_{+_{D}})((\circ_{M^{M}}\otimes
\mathrm{id}_{W_{D(2)}})((\mathrm{id}_{M^{M}}\otimes W_{p_{1}^{D(2)}%
})(Y),(\mathrm{id}_{M^{M}}\otimes W_{p_{2}^{D(2)}})(X)))
\end{align*}

\end{proposition}

\begin{proof}
With due regard to the limit diagram (\ref{3.1.1}) of Weil algebras and the
microlinearity of $M^{M}$, it suffices to see that
\begin{align}
(\mathrm{id}_{M^{M}}\otimes W_{i_{1}^{D(2)}})((\circ_{M^{M}}\otimes
\mathrm{id}_{W_{D(2)}})((\mathrm{id}_{M^{M}}\otimes W_{p_{1}^{D(2)}%
})(X),(\mathrm{id}_{M^{M}}\otimes W_{p_{2}^{D(2)}})(Y)))  &  =X\label{4.5.1}\\
(\mathrm{id}_{M^{M}}\otimes W_{i_{2}^{D(2)}})((\circ_{M^{M}}\otimes
\mathrm{id}_{W_{D(2)}})((\mathrm{id}_{M^{M}}\otimes W_{p_{1}^{D(2)}%
})(X),(\mathrm{id}_{M^{M}}\otimes W_{p_{2}^{D(2)}})(Y)))  &  =Y\label{4.5.2}\\
(\mathrm{id}_{M^{M}}\otimes W_{i_{1}^{D(2)}})((\circ_{M^{M}}\otimes
\mathrm{id}_{W_{D(2)}})((\mathrm{id}_{M^{M}}\otimes W_{p_{1}^{D(2)}%
})(Y),(\mathrm{id}_{M^{M}}\otimes W_{p_{2}^{D(2)}})(X)))  &  =Y\label{4.5.3}\\
(\mathrm{id}_{M^{M}}\otimes W_{i_{2}^{D(2)}})((\circ_{M^{M}}\otimes
\mathrm{id}_{W_{D(2)}})((\mathrm{id}_{M^{M}}\otimes W_{p_{1}^{D(2)}%
})(Y),(\mathrm{id}_{M^{M}}\otimes W_{p_{2}^{D(2)}})(X)))  &  =X\label{4.5.4}%
\end{align}
Here we deal only with (\ref{4.5.1}), leaving similar treatments of
(\ref{4.5.2})-(\ref{4.5.4}) safely to the reader. Exploiting the commutativity
of the diagram (\ref{4.4.3}), we have
\begin{align*}
&  (\mathrm{id}_{M^{M}}\otimes W_{i_{1}^{D(2)}})((\circ_{M^{M}}\otimes
\mathrm{id}_{W_{D(2)}})((\mathrm{id}_{M^{M}}\otimes W_{p_{1}^{D(2)}%
})(X),(\mathrm{id}_{M^{M}}\otimes W_{p_{2}^{D(2)}})(Y)))\\
&  =(\circ_{M^{M}}\otimes\mathrm{id}_{W_{D}})((\mathrm{id}_{M^{M}}\otimes
W_{i_{1}^{D(2)}})\circ(\mathrm{id}_{M^{M}}\otimes W_{p_{1}^{D(2)}})(X),\\
&  (\mathrm{id}_{M^{M}}\otimes W_{i_{1}^{D(2)}})\circ(\mathrm{id}_{M^{M}%
}\otimes W_{p_{2}^{D(2)}})(Y)))\\
&  =(\circ_{M^{M}}\otimes\mathrm{id}_{W_{D}})((\mathrm{id}_{M^{M}}%
\otimes(W_{i_{1}^{D(2)}}\circ W_{p_{1}^{D(2)}}))(X),(\mathrm{id}_{M^{M}%
}\otimes(W_{i_{1}^{D(2)}}\circ W_{p_{2}^{D(2)}})(Y))\\
&  =X
\end{align*}
Thus we are done.
\end{proof}

\begin{theorem}
\label{t4.6}For any $X,Y\in\aleph(M)$, there exists a unique $[X,Y]\in
\aleph(M)$ such that
\begin{align*}
&  (\mathrm{id}_{M^{M}}\otimes W_{(d_{1},d_{2})\in D^{2}\mapsto(d_{1}%
,d_{2},-d_{1},-d_{2})\in D^{4}})(Y\ast X\ast Y\ast X)\\
&  =(\mathrm{id}_{M^{M}}\otimes W_{(d_{1},d_{2})\in D^{2}\mapsto d_{1}d_{2}\in
D})([X,Y])
\end{align*}

\end{theorem}

\begin{proof}
With due regard to the microlinearity of $M^{M}$ and the limit diagram of Weil
algebras
\[
M^{M}\otimes W_{D}
\begin{array}
[c]{c}%
\mathrm{id}_{M^{M}}\otimes W_{i_{1}^{D^{2}}}\\
\longleftarrow\\
\mathrm{id}_{M^{M}}\otimes W_{i_{2}^{D^{2}}}\\
\longleftarrow\\
\mathrm{id}_{M^{M}}\otimes W_{d\in D\mapsto(0,0)\in D^{2}}\\
\longleftarrow\\
\
\end{array}
M^{M}\otimes W_{D^{2}}
\begin{array}
[c]{c}%
\mathrm{id}_{M^{M}}\otimes W_{(d_{1},d_{2})\in D^{2}\mapsto d_{1}d_{2}\in D}\\
\longleftarrow\\
\
\end{array}
M^{M}\otimes W_{D}%
\]
it suffices to see that
\begin{align}
(\mathrm{id}_{M^{M}}\otimes W_{i_{1}^{D^{2}}})(\mathrm{id}_{M^{M}}\otimes
W_{(d_{1},d_{2})\in D^{2}\mapsto(d_{1},d_{2},-d_{1},-d_{2})\in D^{4}}(Y\ast
X\ast Y\ast X))  &  =I_{1}\label{4.6.1}\\
(\mathrm{id}_{M^{M}}\otimes W_{i_{2}^{D^{2}}})(\mathrm{id}_{M^{M}}\otimes
W_{(d_{1},d_{2})\in D^{2}\mapsto(d_{1},d_{2},-d_{1},-d_{2})\in D^{4}}(Y\ast
X\ast Y\ast X))  &  =I_{1}\label{4.6.2}%
\end{align}
Here we deal only with (\ref{4.6.1}), leaving a similar treatment of
(\ref{4.6.2}) safely to the reader. By dint of the bifunctionality of
$\otimes$,\ it is easy to see that
\[%
\begin{array}
[c]{ccc}
& \circ_{M^{M}}\otimes\mathrm{id}_{W_{D^{4}}} & \\
(M^{M}\otimes W_{D^{4}})\times(M^{M}\otimes W_{D4}) & \rightarrow &
M^{M}\otimes W_{D^{4}}\\
\downarrow &  & \downarrow\\
(M^{M}\otimes W_{D})\times(M^{M}\otimes W_{D}) & \rightarrow & M^{M}\otimes
W_{D}\\
& \circ_{M^{M}}\otimes\mathrm{id}_{W_{D}} &
\end{array}
\]
with the left vertical arrow standing for $(\mathrm{id}_{M^{M}}\otimes W_{d\in
D\mapsto(d,0,-d,0)\in D^{4}})\times(\mathrm{id}_{M^{M}}\otimes W_{d\in
D\mapsto(d,0,-d,0)\in D^{4}})$ and the right vertical arrow standing for
$\mathrm{id}_{M^{M}}\otimes W_{d\in D\mapsto(d,0,-d,0)\in D^{4}}$, so that we
have
\begin{align*}
&  (\mathrm{id}_{M^{M}}\otimes W_{d\in D\mapsto(d,0,-d,0)\in D^{4}})(Y\ast
X\ast Y\ast X))\\
&  =(\circ_{M^{M}}\otimes\mathrm{id}_{W_{D}})((\mathrm{id}_{M^{M}}\otimes
W_{d\in D\mapsto(d,0)\in D^{2}})(Y\ast X),(\mathrm{id}_{M^{M}}\otimes W_{d\in
D\mapsto(-d,0)\in D^{2}})(Y\ast X))
\end{align*}
Therefore we have
\begin{align*}
&  (\mathrm{id}_{M^{M}}\otimes W_{i_{1}^{D^{2}}})(\mathrm{id}_{M^{M}}\otimes
W_{(d_{1},d_{2})\in D^{2}\mapsto(d_{1},d_{2},d_{1},d_{2})\in D^{4}}(Y\ast
X\ast Y\ast X))\\
&  =(\mathrm{id}_{M^{M}}\otimes W_{d\in D\mapsto(d,0,-d,0)\in D^{4}})(Y\ast
X\ast Y\ast X))\\
&  =(\circ_{M^{M}}\otimes\mathrm{id}_{W_{D}})((\mathrm{id}_{M^{M}}\otimes
W_{d\in D\mapsto(d,0)\in D^{2}})(Y\ast X),(\mathrm{id}_{M^{M}}\otimes W_{d\in
D\mapsto(-d,0)\in D^{2}})(Y\ast X))\\
&  =(\circ_{M^{M}}\otimes\mathrm{id}_{W_{D}})(X,-X)\\
&  =\mathbf{0}_{\mathrm{id}_{M}}%
\end{align*}
Thus we are done.
\end{proof}

\begin{proposition}
\label{t4.7}For any $X,Y\in\aleph(M)$, we have
\[
\lbrack X,Y]=-[Y,X]
\]

\end{proposition}

\begin{proof}
With due regard to Proposition \ref{t4.5}, it suffices to show that
\[
(\mathrm{id}_{M^{M}}\otimes W_{(d_{1},d_{2})\in D^{2}\mapsto(d_{1}d_{2}%
,d_{1}d_{2})\in D^{2}})([X,Y]\ast\lbrack Y,X])=I_{2}%
\]
This follows from
\begin{align*}
&  (\mathrm{id}_{M^{M}}\otimes W_{(d_{1},d_{2})\in D^{2}\mapsto(d_{1}%
d_{2},d_{1}d_{2})\in D^{2}})([X,Y]\ast\lbrack Y,X])\\
&  =(\mathrm{id}_{M^{M}}\otimes W_{(d_{1},d_{2})\in D^{2}\mapsto(d_{2}%
,d_{1},-d_{2},-d_{1},d_{1},d_{2},-d_{1},-d_{2})\in D^{8}})(Y\ast X\ast Y\ast
X\ast X\ast Y\ast X\ast Y)\\
&  =(\circ_{M^{M}}\otimes\mathrm{id}_{W_{D^{2}}})((\mathrm{id}_{M^{M}}\otimes
W_{(d_{1},d_{2})\in D^{2}\mapsto(d_{2},d_{1},-d_{2})\in D^{3}})(Y\ast X\ast
Y),\\
&  (\mathrm{id}_{M^{M}}\otimes W_{(d_{1},d_{2})\in D^{2}\mapsto(-d_{1}%
,d_{1})\in D^{2}})(X\ast X),\\
&  (\mathrm{id}_{M^{M}}\otimes W_{(d_{1},d_{2})\in D^{2}\mapsto(d_{2}%
,-d_{1},-d_{2})\in D^{3}})(Y\ast X\ast Y))\\
&  \text{[by the bifunctionality of }\otimes\text{]}\\
&  =(\circ_{M^{M}}\otimes\mathrm{id}_{W_{D^{2}}})((\mathrm{id}_{M^{M}}\otimes
W_{(d_{1},d_{2})\in D^{2}\mapsto(d_{2},d_{1},-d_{2})\in D^{3}})(Y\ast X\ast
Y),I_{2},\\
&  (\mathrm{id}_{M^{M}}\otimes W_{(d_{1},d_{2})\in D^{2}\mapsto(d_{2}%
,-d_{1},-d_{2})\in D^{3}})(Y\ast X\ast Y))\\
&  =(\circ_{M^{M}}\otimes\mathrm{id}_{W_{D^{2}}})((\mathrm{id}_{M^{M}}\otimes
W_{(d_{1},d_{2})\in D^{2}\mapsto(d_{2},d_{1},-d_{2})\in D^{3}})(Y\ast X\ast
Y),\\
&  (\mathrm{id}_{M^{M}}\otimes W_{(d_{1},d_{2})\in D^{2}\mapsto(d_{2}%
,-d_{1},-d_{2})\in D^{3}})(Y\ast X\ast Y))\\
&  =(\mathrm{id}_{M^{M}}\otimes W_{(d_{1},d_{2})\in D^{2}\mapsto(d_{2}%
,d_{1},-d_{2},d_{2},-d_{1},-d_{2})\in D^{6}})(Y\ast X\ast Y\ast Y\ast X\ast
Y)\\
&  \text{[by the bifunctionality of }\otimes\text{]}\\
&  =(\circ_{M^{M}}\otimes\mathrm{id}_{W_{D}})((\mathrm{id}_{M^{M}}\otimes
W_{(d_{1},d_{2})\in D^{2}\mapsto(d_{2},d_{1})\in D^{2}})(X\ast Y),\\
&  (\mathrm{id}_{M^{M}}\otimes W_{(d_{1},d_{2})\in D^{2}\mapsto(-d_{2}%
,d_{2})\in D^{2}})(Y\ast Y),(\mathrm{id}_{M^{M}}\otimes W_{(d_{1},d_{2})\in
D^{2}\mapsto(-d_{1},-d_{2})\in D^{2}})(Y\ast X))\\
&  \text{[by the bifunctionality of }\otimes\text{]}\\
&  =(\circ_{M^{M}}\otimes\mathrm{id}_{W_{D^{2}}})((\mathrm{id}_{M^{M}}\otimes
W_{(d_{1},d_{2})\in D^{2}\mapsto(d_{2},d_{1})\in D^{2}})(X\ast Y),I_{2},\\
&  (\mathrm{id}_{M^{M}}\otimes W_{(d_{1},d_{2})\in D^{2}\mapsto(-d_{1}%
,-d_{2})\in D^{2}})(Y\ast X))\\
&  =(\circ_{M^{M}}\otimes\mathrm{id}_{W_{D^{2}}})((\mathrm{id}_{M^{M}}\otimes
W_{(d_{1},d_{2})\in D^{2}\mapsto(d_{2},d_{1})\in D^{2}})(X\ast Y),\\
&  (\mathrm{id}_{M^{M}}\otimes W_{(d_{1},d_{2})\in D^{2}\mapsto(-d_{1}%
,-d_{2})\in D^{2}})(Y\ast X))\\
&  =(\mathrm{id}_{M^{M}}\otimes W_{(d_{1},d_{2})\in D^{2}\mapsto(d_{2}%
,d_{1},-d_{1},-d_{2})\in D^{4}})(Y\ast X\ast X\ast Y)\\
&  \text{[by the bifunctionality of }\otimes\text{]}\\
&  =(\circ_{M^{M}}\otimes\mathrm{id}_{W_{D^{2}}})((\mathrm{id}_{M^{M}}\otimes
W_{(d_{1},d_{2})\in D^{2}\mapsto d_{2}\in D})(Y),(\mathrm{id}_{M^{M}}\otimes
W_{(d_{1},d_{2})\in D^{2}\mapsto(d_{1},-d_{1})\in D^{2}})(X\ast X),\\
&  (\mathrm{id}_{M^{M}}\otimes W_{(d_{1},d_{2})\in D^{2}\mapsto-d_{2}\in
D})(Y))\\
&  \text{[by the bifunctionality of }\otimes\text{]}\\
&  =(\circ_{M^{M}}\otimes\mathrm{id}_{W_{D^{2}}})((\mathrm{id}_{M^{M}}\otimes
W_{(d_{1},d_{2})\in D^{2}\mapsto d_{2}\in D})(Y),I_{2},\\
&  (\mathrm{id}_{M^{M}}\otimes W_{(d_{1},d_{2})\in D^{2}\mapsto-d_{2}\in
D})(Y))\\
&  =(\circ_{M^{M}}\otimes\mathrm{id}_{W_{D^{2}}})((\mathrm{id}_{M^{M}}\otimes
W_{(d_{1},d_{2})\in D^{2}\mapsto d_{2}\in D})(Y),(\mathrm{id}_{M^{M}}\otimes
W_{(d_{1},d_{2})\in D^{2}\mapsto-d_{2}\in D})(Y))\\
&  =(\mathrm{id}_{M^{M}}\otimes W_{(d_{1},d_{2})\in D^{2}\mapsto(d_{2}%
,-d_{2})\in D^{2}})(Y\ast Y)\\
&  =I_{2}%
\end{align*}

\end{proof}

The proof of the Jacobi identity is postponed to the subsequent two sections.

\section{The General Jacobi Identity}

The principal objective in this section is to give a proof of the general
Jacobi identity. Our harder treatment of the general Jacobi identity is
preceded by a simpler treatment of the primordial Jacobi identity, because the
latter is easier to grasp intuitively.

\begin{proposition}
\label{t5.1}The diagram
\[%
\begin{array}
[c]{ccccc}
&  & \mathrm{id}_{M}\otimes W_{\varphi} &  & \\
& M\otimes W_{D^{3}\{(1,3),(2,3)\}} & \rightarrow & M\otimes W_{D^{2}} & \\
\mathrm{id}_{M}\otimes W_{\psi} & \downarrow &  & \downarrow & \mathrm{id}%
_{M}\otimes W_{i_{D(2)}^{D^{2}}}\\
& M\otimes W_{D^{2}} & \rightarrow & M\otimes W_{D(2)} & \\
&  & \mathrm{id}_{M}\otimes W_{i_{D(2)}^{D^{2}}} &  &
\end{array}
\]
is a pullback diagram, where the putative mapping $\varphi:D^{2}\rightarrow
D^{3}\{(1,3),(2,3)\}$ is
\[
(d_{1},d_{2})\in D^{2}\mapsto(d_{1},d_{2},0)\in D^{3}\{(1,3),(2,3)\}
\]
while the putative mapping $\psi:D^{2}\rightarrow D^{3}\{(1,3),(2,3)\}$ is
\[
(d_{1},d_{2})\in D^{2}\mapsto(d_{1},d_{2},d_{1}d_{2})\in D^{3}\{(1,3),(2,3)\}
\]

\end{proposition}

\begin{proof}
This follows from the microlinearity of $M$ and the pullback diagram of Weil
algebras
\[%
\begin{array}
[c]{ccccc}
&  & W_{\varphi} &  & \\
& W_{D^{3}\{(1,3),(2,3)\}} & \rightarrow & W_{D^{2}} & \\
W_{\psi} & \downarrow &  & \downarrow & W_{i_{D(2)}^{D^{2}}}\\
& W_{D^{2}} & \rightarrow & W_{D(2)} & \\
&  & W_{i_{D(2)}^{D^{2}}} &  &
\end{array}
\]

\end{proof}

\begin{corollary}
For any $\gamma_{1},\gamma_{2}\in M\otimes W_{D^{2}}$, if $\left(
\mathrm{id}_{M}\otimes W_{i_{D(2)}^{D^{2}}}\right)  (\gamma_{1})=\left(
\mathrm{id}_{M}\otimes W_{i_{D(2)}^{D^{2}}}\right)  (\gamma_{2})$, then there
exists unique $\gamma\in M\otimes W_{D^{3}\{(1,3),(2,3)\}}$ with
$(\mathrm{id}_{M}\otimes W_{\varphi})(\gamma)=\gamma_{1}$ and $(\mathrm{id}%
_{M}\otimes W_{\psi})(\gamma)=\gamma_{2}$.
\end{corollary}

\begin{remark}
Thus $\gamma$ encodes $\gamma_{1}$ and $\gamma_{2}$, which are in turn
recovered from $\gamma$ via $\mathrm{id}_{M}\otimes W_{\varphi}$ and
$\mathrm{id}_{M}\otimes W_{\psi}$ respectively.
\end{remark}

\begin{notation}
We will write $g_{(\gamma_{1},\gamma_{2})}$ for $\gamma$ in the above corollary.
\end{notation}

\begin{definition}
The \textit{strong difference} $\gamma_{2}\overset{\cdot}{-}\gamma_{1}\in
M\otimes W_{D}$ is defined to be
\[
(\mathrm{id}_{M}\otimes W_{d\in D\mapsto(0,0,d)\in D^{3}\{(1,3),(2,3)\}}%
)(g_{(\gamma_{1},\gamma_{2})})
\]

\end{definition}

The following is the prototype for the general Jacobi identity.

\begin{theorem}
\label{t5.2}(The Primordial Jacobi Identity) Let $\gamma_{1},\gamma_{2}%
,\gamma_{3}\in M\otimes W_{D^{2}}$. As long as the following three expressions
are well defined (i.e., $\left(  \mathrm{id}_{M}\otimes W_{i_{D(2)}^{D^{2}}%
}\right)  (\gamma_{1})=\left(  \mathrm{id}_{M}\otimes W_{i_{D(2)}^{D^{2}}%
}\right)  (\gamma_{2})=\left(  \mathrm{id}_{M}\otimes W_{i_{D(2)}^{D^{2}}%
}\right)  (\gamma_{3})$), they sum up only to vanish:
\begin{align*}
&  \gamma_{2}\overset{\cdot}{-}\gamma_{1}\\
&  \gamma_{3}\overset{\cdot}{-}\gamma_{2}\\
&  \gamma_{1}\overset{\cdot}{-}\gamma_{3}%
\end{align*}

\end{theorem}

Its proof is based completely upon the following theorem.

\begin{theorem}
\label{t5.3}The diagram
\[%
\begin{array}
[c]{ccccccc}
& \mathrm{id}_{M}\otimes W_{i_{D(2)}^{D^{2}}} &  & M\otimes W_{D^{2}} &  &
\mathrm{id}_{M}\otimes W_{i_{D(2)}^{D^{2}}} & \\
&  & \swarrow & \uparrow & \searrow &  & \\
& M\otimes W_{D(2)} &  & M\otimes W_{E} &  & M\otimes W_{D(2)} & \\
\mathrm{id}_{M}\otimes W_{i_{D(2)}^{D^{2}}} & \uparrow & \swarrow &  &
\searrow & \uparrow & \mathrm{id}_{M}\otimes W_{i_{D(2)}^{D^{2}}}\\
& M\otimes W_{D^{2}} &  &  &  & M\otimes W_{D^{2}} & \\
&  & \searrow &  & \swarrow &  & \\
& \mathrm{id}_{M}\otimes W_{i_{D(2)}^{D^{2}}} &  & M\otimes W_{D(2)} &  &
\mathrm{id}_{M}\otimes W_{i_{D(2)}^{D^{2}}} &
\end{array}
\]
is a limit diagram, where the putative object $E$ is
\[
D^{4}\{(1,3),(2,3),(1,4),(2,4),(3,4)\}
\]
and the putative mapping $i_{D(2)}^{D^{2}}:D(2)\rightarrow D^{2}$ is
$(d_{1},d_{2})\in D(2)\mapsto(d_{1},d_{2})\in D^{2}$, while the three unnamed
arrows $M\otimes W_{E}\rightarrow M\otimes W_{D^{2}}$ are $\mathrm{id}%
_{M}\otimes W_{l_{i}}$ $(i=1,2,3)$ with the putative mappings $l_{i}%
:D^{2}\rightarrow E$ $(i=1,2,3)$ being
\begin{align*}
l_{1}  &  :(d_{1},d_{2})\in D^{2}\mapsto(d_{1},d_{2},0,0)\in E\\
l_{2}  &  :(d_{1},d_{2})\in D^{2}\mapsto(d_{1},d_{2},d_{1}d_{2},0)\in E\\
l_{3}  &  :(d_{1},d_{2})\in D^{2}\mapsto(d_{1},d_{2},0,d_{1}d_{2})\in E
\end{align*}

\end{theorem}

This theorem follows directly from the following lemma.

\begin{lemma}
\label{t5.4}The following diagram is a limit diagram of Weil algebras:
\[
\
\begin{array}
[c]{ccccccc}
& W_{i_{D(2)}^{D^{2}}} &  & W_{D^{2}} &  & W_{i_{D(2)}^{D^{2}}} & \\
&  & \swarrow & \uparrow & \searrow &  & \\
& W_{D(2)} &  & W_{E} &  & W_{D(2)} & \\
W_{i_{D(2)}^{D^{2}}} & \uparrow & \swarrow &  & \searrow & \uparrow &
W_{i_{D(2)}^{D^{2}}}\\
& W_{D^{2}} &  &  &  & W_{D^{2}} & \\
&  & \searrow &  & \swarrow &  & \\
& W_{i_{D(2)}^{D^{2}}} &  & W_{D(2)} &  & W_{i_{D(2)}^{D^{2}}} &
\end{array}
\]

\end{lemma}

\begin{proof}
Let $\gamma_{1},\gamma_{2},\gamma_{3}\in W_{D^{2}}$ and $\gamma\in W_{E}$ so
that they are the polynomials with real coefficients in the following form:
\begin{align*}
\gamma_{1}(X_{1},X_{2})  &  =a+a_{1}X_{1}+a_{2}X_{2}+a_{12}X_{1}X_{2}\\
\gamma_{2}(X_{1},X_{2})  &  =b+b_{1}X_{1}+b_{2}X_{2}+b_{12}X_{1}X_{2}\\
\gamma_{3}(X_{1},X_{2})  &  =c+c_{1}X_{1}+c_{2}X_{2}+c_{12}X_{1}X_{2}\\
\gamma(X_{1},X_{2},X_{3},X_{4})  &  =e+e_{1}X_{1}+e_{2}X_{2}+e_{12}X_{1}%
X_{2}+e_{3}X_{3}+e_{4}X_{4}%
\end{align*}
The condition that $W_{i_{D(2)}^{D^{2}}}(\gamma_{1})=W_{i_{D(2)}^{D^{2}}%
}(\gamma_{2})=W_{i_{D(2)}^{D^{2}}}(\gamma_{3})$ is equivalent to the following
three conditions as a whole:
\begin{align*}
a  &  =b=c\\
a_{1}  &  =b_{1}=c_{1}\\
a_{2}  &  =b_{2}=c_{2}%
\end{align*}
Therefore, in order that $W_{l_{1}}(\gamma)=\gamma_{1}$, $W_{l_{2}}%
(\gamma)=\gamma_{2}$ and $W_{l_{3}}(\gamma)=\gamma_{3}$ in this case, it is
necessary and sufficient that the polynomial $\gamma$ should be of the
following form:
\[
\gamma(X_{1},X_{2},X_{3},X_{4})=a+a_{1}X_{1}+a_{2}X_{2}+a_{12}X_{1}%
X_{2}+(b_{12}-a_{12})X_{3}+(c_{12}-a_{12})X_{4}%
\]
This completes the proof.
\end{proof}

\begin{corollary}
Given $\gamma_{1},\gamma_{2},\gamma_{3}\in M\otimes W_{D^{2}}$ with
$W_{i_{D(2)}^{D^{2}}}(\gamma_{1})=W_{i_{D(2)}^{D^{2}}}(\gamma_{2}%
)=W_{i_{D(2)}^{D^{2}}}(\gamma_{3})$, there exists a unique $\gamma\in M\otimes
W_{E}$, usally denoted by $h_{(\gamma_{1},\gamma_{2},\gamma_{3})} $, such that
$\gamma_{i}=(\mathrm{id}_{M}\otimes W_{l_{i}})(\gamma)$ $(i=1,2,3)$.
\end{corollary}

\begin{remark}
Thus $h_{(\gamma_{1},\gamma_{2},\gamma_{3})}$ encodes $\gamma_{1}$,
$\gamma_{2}$ and $\gamma_{3}$, which are in turn recovered from $h_{(\gamma
_{1},\gamma_{2},\gamma_{3})}$ via $\mathrm{id}_{M}\otimes W_{l_{i}}$
$(i=1,2,3)$.
\end{remark}

\begin{proof}
(of the primordial Jacobi identity). Let $t_{i}$ $(i=1,2,3)$ be the three
expression in Theorem \ref{t5.2} in order. It is easy to see that
\begin{align*}
g_{(\gamma_{1},\gamma_{2})}  &  =(\mathrm{id}_{M}\otimes W_{(d_{1},d_{2}%
,d_{3})\in D^{3}\{(1,3),(2,3)\}\mapsto(d_{1},d_{2},d_{3},0)\in E}%
)(h_{(\gamma_{1},\gamma_{2},\gamma_{3})})\\
g_{(\gamma_{2},\gamma_{3})}  &  =(\mathrm{id}_{M}\otimes W_{(d_{1},d_{2}%
,d_{3})\in D^{3}\{(1,3),(2,3)\}\mapsto(d_{1},d_{2},d_{1}d_{2}-d_{3},d_{3})\in
E})(h_{(\gamma_{1},\gamma_{2},\gamma_{3})})\\
g_{(\gamma_{3},\gamma_{1})}  &  =(\mathrm{id}_{M}\otimes W_{(d_{1},d_{2}%
,d_{3})\in D^{3}\{(1,3),(2,3)\}\mapsto(d_{1},d_{2},0,d_{1}d_{2}-d_{3})\in
E})(h_{(\gamma_{1},\gamma_{2},\gamma_{3})})
\end{align*}
Therefore we have
\begin{align*}
&  t_{1}\\
&  =(\mathrm{id}_{M}\otimes W_{d\in D\mapsto(0,0,d)\in D^{3}\{(1,3),(2,3)\}}%
)(g_{(\gamma_{1},\gamma_{2})})\\
&  =(\mathrm{id}_{M}\otimes W_{d\in D\mapsto(0,0,d)\in D^{3}\{(1,3),(2,3)\}}%
)\circ(\mathrm{id}_{M}\otimes W_{(d_{1},d_{2},d_{3})\in D^{3}%
\{(1,3),(2,3)\}\mapsto(d_{1},d_{2},d_{3},0)\in E})\\
&  (h_{(\gamma_{1},\gamma_{2},\gamma_{3})})\\
&  =(\mathrm{id}_{M}\otimes W_{d\in D\mapsto(0,0,d,0)\in E})(h_{(\gamma
_{1},\gamma_{2},\gamma_{3})})
\end{align*}
\begin{align*}
&  t_{2}\\
&  =(\mathrm{id}_{M}\otimes W_{d\in D\mapsto(0,0,d)\in D^{3}\{(1,3),(2,3)\}}%
)(g_{(\gamma_{2},\gamma_{3})})\\
&  =(\mathrm{id}_{M}\otimes W_{d\in D\mapsto(0,0,d)\in D^{3}\{(1,3),(2,3)\}}%
)\circ(\mathrm{id}_{M}\otimes W_{(d_{1},d_{2},d_{3})\in D^{3}%
\{(1,3),(2,3)\}\mapsto(d_{1},d_{2},d_{1}d_{2}-d_{3},d_{3})\in E})\\
&  (h_{(\gamma_{1},\gamma_{2},\gamma_{3})})\\
&  =(\mathrm{id}_{M}\otimes W_{d\in D\mapsto(0,0,-d,d)\in E})(h_{(\gamma
_{1},\gamma_{2},\gamma_{3})})
\end{align*}
\begin{align*}
&  t_{3}\\
&  =(\mathrm{id}_{M}\otimes W_{d\in D\mapsto(0,0,d)\in D^{3}\{(1,3),(2,3)\}}%
)(g_{(\gamma_{2},\gamma_{3})})\\
&  =(\mathrm{id}_{M}\otimes W_{d\in D\mapsto(0,0,d)\in D^{3}\{(1,3),(2,3)\}}%
)\circ(\mathrm{id}_{M}\otimes W_{(d_{1},d_{2},d_{3})\in D^{3}%
\{(1,3),(2,3)\}\mapsto(d_{1},d_{2},0,d_{1}d_{2}-d_{3})\in E})\\
&  (h_{(\gamma_{1},\gamma_{2},\gamma_{3})})\\
&  =(\mathrm{id}_{M}\otimes W_{d\in D\mapsto(0,0,0,-d)\in E})(h_{(\gamma
_{1},\gamma_{2},\gamma_{3})})
\end{align*}
Thus we have
\[
l_{(t_{1},t_{2},t_{3})}=(\mathrm{id}_{M}\otimes W_{(d_{1},d_{2},d_{3})\in
D(3)\mapsto(0,0,d_{1}-d_{2},d_{2}-d_{3})\in E})(h_{(\gamma_{1},\gamma
_{2},\gamma_{3})})
\]
This means that
\begin{align*}
&  t_{1}+t_{2}+t_{3}\\
&  =(\mathrm{id}_{M}\otimes W_{d\in D\mapsto(d,d,d)\in D(3)})(l_{(t_{1}%
,t_{2},t_{3})})\\
&  =(\mathrm{id}_{M}\otimes W_{d\in D\mapsto(d,d,d)\in D(3)})\circ
(\mathrm{id}_{M}\otimes W_{(d_{1},d_{2},d_{3})\in D(3)\mapsto(0,0,d_{1}%
-d_{2},d_{2}-d_{3})\in E})(h_{(\gamma_{1},\gamma_{2},\gamma_{3})})\\
&  =(\mathrm{id}_{M}\otimes W_{d\in D\mapsto(0,0,d-d,d-d)\in E})(h_{(\gamma
_{1},\gamma_{2},\gamma_{3})})\\
&  =(\mathrm{id}_{M}\otimes W_{d\in D\mapsto(0,0,0,0)\in E})(h_{(\gamma
_{1},\gamma_{2},\gamma_{3})})
\end{align*}
Thus the proof of the primordial Jacobi identity is complete.
\end{proof}

\begin{proposition}
\label{t5.5}The diagram
\[%
\begin{array}
[c]{ccccc}
&  & \mathrm{id}_{M}\otimes W_{\varphi_{1}^{3}} &  & \\
& M\otimes W_{D^{4}\{(2,4),(3,4)\}} & \rightarrow & M\otimes W_{D^{3}} & \\
\mathrm{id}_{M}\otimes W_{\psi_{1}^{3}} & \downarrow &  & \downarrow &
\mathrm{id}_{M}\otimes W_{i_{D^{3}\{(2,3)\}}^{D^{3}}}\\
& M\otimes W_{D^{3}} & \rightarrow & M\otimes W_{D^{3}\{(2,3)\}} & \\
&  & \mathrm{id}_{M}\otimes W_{i_{D^{3}\{(2,3)\}}^{D^{3}}} &  &
\end{array}
\]
is a pullback diagram, where the putative mapping $\varphi_{1}^{3}%
:D^{3}\rightarrow D^{4}\{(2,4),(3,4)\}$ is
\[
(d_{1},d_{2},d_{3})\in D^{3}\mapsto(d_{1},d_{2},d_{3},0)\in D^{4}%
\{(2,4),(3,4)\}
\]
while the putative mapping $\psi_{1}^{3}:D^{3}\rightarrow D^{4}%
\{(2,4),(3,4)\}$ is
\[
(d_{1},d_{2},d_{3})\in D^{3}\mapsto(d_{1},d_{2},d_{3},d_{2}d_{3})\in
D^{4}\{(2,4),(3,4)\}
\]

\end{proposition}

\begin{proof}
This follows from the microlinearity of $M$ and the pullback diagram of Weil
algebras
\[%
\begin{array}
[c]{ccccc}
&  & W_{\varphi_{1}^{3}} &  & \\
& W_{D^{4}\{(2,4),(3,4)\}} & \rightarrow & W_{D^{3}} & \\
W_{\psi_{1}^{3}} & \downarrow &  & \downarrow & W_{i_{D^{3}\{(2,3)\}}^{D^{3}}%
}\\
& W_{D^{3}} & \rightarrow & W_{D^{3}\{(2,3)\}} & \\
&  & W_{i_{D^{3}\{(2,3)\}}^{D^{3}}} &  &
\end{array}
\]

\end{proof}

\begin{corollary}
For any $\gamma_{1},\gamma_{2}\in M\otimes W_{D^{3}}$, if $\left(
\mathrm{id}_{M}\otimes W_{i_{D^{3}\{(2,3)\}}^{D^{3}}}\right)  (\gamma
_{1})=\left(  \mathrm{id}_{M}\otimes W_{i_{D^{3}\{(2,3)\}}^{D^{3}}}\right)
(\gamma_{2})$, then there exists unique $\gamma\in M\otimes W_{D^{4}%
\{(2,4),(3,4)\}}$ with $(\mathrm{id}_{M}\otimes W_{\varphi_{1}^{3}}%
)(\gamma)=\gamma_{1}$ and $(\mathrm{id}_{M}\otimes W_{\psi_{1}^{3}}%
)(\gamma)=\gamma_{2}$.
\end{corollary}

\begin{remark}
Thus $\gamma$ encodes $\gamma_{1}$ and $\gamma_{2}$, which are in turn
recovered from $\gamma$ via $\mathrm{id}_{M}\otimes W_{\varphi_{1}^{3}}$ and
$\mathrm{id}_{M}\otimes W_{\psi_{1}^{3}}$ respectively.
\end{remark}

\begin{notation}
We will write $g_{(\gamma_{1},\gamma_{2})}^{1}$ for $\gamma$ in the above corollary.
\end{notation}

\begin{definition}
The (\textit{first)} \textit{strong difference} $\gamma_{2}\underset
{1}{\overset{\cdot}{-}}\gamma_{1}\in M\otimes W_{D^{2}}$ is defined to be
\[
(\mathrm{id}_{M}\otimes W_{(d_{1},d_{2})\in D^{2}\mapsto(d_{1},0,0,d_{2})\in
D^{4}\{(2,4),(3,4)\}})(g_{(\gamma_{1},\gamma_{2})}^{1})
\]

\end{definition}

\begin{proposition}
\label{t5.6}The diagram
\[%
\begin{array}
[c]{ccccc}
&  & \mathrm{id}_{M}\otimes W_{\varphi_{2}^{3}} &  & \\
& M\otimes W_{D^{4}\{(1,4),(3,4)\}} & \rightarrow & M\otimes W_{D^{3}} & \\
\mathrm{id}_{M}\otimes W_{\psi_{2}^{3}} & \downarrow &  & \downarrow &
\mathrm{id}_{M}\otimes W_{i_{D^{3}\{(1,3)\}}^{D^{3}}}\\
& M\otimes W_{D^{3}} & \rightarrow & M\otimes W_{D^{3}\{(1,3)\}} & \\
&  & \mathrm{id}_{M}\otimes W_{i_{D^{3}\{(1,3)\}}^{D^{3}}} &  &
\end{array}
\]
is a pullback diagram, where the putative mapping $\varphi_{2}^{3}%
:D^{3}\rightarrow D^{4}\{(1,4),(3,4)\}$ is
\[
(d_{1},d_{2},d_{3})\in D^{3}\mapsto(d_{1},d_{2},d_{3},0)\in D^{4}%
\{(1,4),(3,4)\}
\]
while the putative mapping $\psi_{2}^{3}:D^{3}\rightarrow D^{4}%
\{(1,4),(3,4)\}$ is
\[
(d_{1},d_{2},d_{3})\in D^{3}\mapsto(d_{1},d_{2},d_{3},d_{1}d_{3})\in
D^{4}\{(1,4),(3,4)\}
\]

\end{proposition}

\begin{proof}
This follows from the microlinearity of $M$ and the pullback diagram of Weil
algebras
\[%
\begin{array}
[c]{ccccc}
&  & W_{\varphi_{2}^{3}} &  & \\
& W_{D^{4}\{(1,4),(3,4)\}} & \rightarrow & W_{D^{3}} & \\
W_{\psi_{2}^{3}} & \downarrow &  & \downarrow & W_{i_{D^{3}\{(1,3)\}}^{D^{3}}%
}\\
& W_{D^{3}} & \rightarrow & W_{D^{3}\{(1,3)\}} & \\
&  & W_{i_{D^{3}\{(1,3)\}}^{D^{3}}} &  &
\end{array}
\]

\end{proof}

\begin{corollary}
For any $\gamma_{1},\gamma_{2}\in M\otimes W_{D^{3}}$, if $\left(
\mathrm{id}_{M}\otimes W_{i_{D^{3}\{(1,3)\}}^{D^{3}}}\right)  (\gamma
_{1})=\left(  \mathrm{id}_{M}\otimes W_{i_{D^{3}\{(1,3)\}}^{D^{3}}}\right)
(\gamma_{2})$, then there exists unique $\gamma\in M\otimes W_{D^{4}%
\{(1,4),(3,4)\}}$ with $(\mathrm{id}_{M}\otimes W_{\varphi_{2}^{3}}%
)(\gamma)=\gamma_{1}$ and $(\mathrm{id}_{M}\otimes W_{\psi_{2}^{3}}%
)(\gamma)=\gamma_{2}$.
\end{corollary}

\begin{remark}
Thus $\gamma$ encodes $\gamma_{1}$ and $\gamma_{2}$, which are in turn
recovered from $\gamma$ via $\mathrm{id}_{M}\otimes W_{\varphi_{2}^{3}}$ and
$\mathrm{id}_{M}\otimes W_{\psi_{2}^{3}}$.
\end{remark}

\begin{notation}
We will write $g_{(\gamma_{1},\gamma_{2})}^{2}$ for $\gamma$ in the above corollary.
\end{notation}

\begin{definition}
The\textit{\ (second) strong difference} $\gamma_{2}\underset{2}%
{\overset{\cdot}{-}}\gamma_{1}\in M\otimes W_{D^{2}}$ is defined to be
\[
(\mathrm{id}_{M}\otimes W_{(d_{1},d_{2})\in D^{2}\mapsto(0,d_{1},0,d_{2})\in
D^{4}\{(1,4),(3,4)\}})(g_{(\gamma_{1},\gamma_{2})}^{2})
\]

\end{definition}

\begin{proposition}
\label{t5.7}The diagram
\[%
\begin{array}
[c]{ccccc}
&  & \mathrm{id}_{M}\otimes W_{\varphi_{3}^{3}} &  & \\
& M\otimes W_{D^{4}\{(1,4),(2,4)\}} & \rightarrow & M\otimes W_{D^{3}} & \\
\mathrm{id}_{M}\otimes W_{\psi_{3}^{3}} & \downarrow &  & \downarrow &
\mathrm{id}_{M}\otimes W_{i_{D^{3}\{(1,2)\}}^{D^{3}}}\\
& M\otimes W_{D^{3}} & \rightarrow & M\otimes W_{D^{3}\{(1,2)\}} & \\
&  & \mathrm{id}_{M}\otimes W_{i_{D^{3}\{(1,2)\}}^{D^{3}}} &  &
\end{array}
\]
is a pullback diagram, where the putative mapping $\varphi_{3}^{3}%
:D^{3}\rightarrow D^{4}\{(1,4),(2,4)\}$ is
\[
(d_{1},d_{2},d_{3})\in D^{3}\mapsto(d_{1},d_{2},d_{3},0)\in D^{4}%
\{(1,4),(2,4)\}
\]
while the putative mapping $\psi_{3}^{3}:D^{3}\rightarrow D^{4}%
\{(1,4),(2,4)\}$ is
\[
(d_{1},d_{2},d_{3})\in D^{3}\mapsto(d_{1},d_{2},d_{3},d_{1}d_{2})\in
D^{4}\{(1,4),(2,4)\}
\]

\end{proposition}

\begin{proof}
This follows from the microlinearity of $M$ and the pullback diagram of Weil
algebras
\[%
\begin{array}
[c]{ccccc}
&  & W_{\varphi_{3}^{3}} &  & \\
& W_{D^{4}\{(1,4),(2,4)\}} & \rightarrow & W_{D^{3}} & \\
W_{\psi_{3}^{3}} & \downarrow &  & \downarrow & W_{i_{D^{3}\{(1,2)\}}^{D^{3}}%
}\\
& W_{D^{3}} & \rightarrow & W_{D^{3}\{(1,2)\}} & \\
&  & W_{i_{D^{3}\{(1,2)\}}^{D^{3}}} &  &
\end{array}
\]

\end{proof}

\begin{corollary}
For any $\gamma_{1},\gamma_{2}\in M\otimes W_{D^{3}}$, if $\left(
\mathrm{id}_{M}\otimes W_{i_{D^{3}\{(1,2)\}}^{D^{3}}}\right)  (\gamma
_{1})=\left(  \mathrm{id}_{M}\otimes W_{i_{D^{3}\{(1,2)\}}^{D^{3}}}\right)
(\gamma_{2})$, then there exists unique $\gamma\in M\otimes W_{D^{4}%
\{(1,4),(2,4)\}}$ with $(\mathrm{id}_{M}\otimes W_{\varphi_{3}^{3}}%
)(\gamma)=\gamma_{1}$ and $(\mathrm{id}_{M}\otimes W_{\psi_{3}^{3}}%
)(\gamma)=\gamma_{2}$.
\end{corollary}

\begin{remark}
Thus $\gamma$ encodes $\gamma_{1}$ and $\gamma_{2}$, which are in turn
recovered from $\gamma$ via $\mathrm{id}_{M}\otimes W_{\varphi_{3}^{3}}$ and
$\mathrm{id}_{M}\otimes W_{\psi_{3}^{3}}$.
\end{remark}

\begin{notation}
We will write $g_{(\gamma_{1},\gamma_{2})}^{3}$ for $\gamma$ in the above corollary.
\end{notation}

\begin{definition}
The (\textit{third) strong difference} $\gamma_{2}\underset{3}{\overset{\cdot
}{-}}\gamma_{1}\in M\otimes W_{D^{2}}$ is defined to be
\[
(\mathrm{id}_{M}\otimes W_{(d_{1},d_{2})\in D^{2}\mapsto(0,0,d_{1},d_{2})\in
D^{4}\{(1,4),(3,4)\}})(g_{(\gamma_{1},\gamma_{2})}^{3})
\]

\end{definition}

The general Jacobi identity discovered by Nishimura \cite{nishi-a} goes as follows:

\begin{theorem}
\label{t5.8}(The General Jacobi Identity) Let $\gamma_{123},\gamma
_{132},\gamma_{213},\gamma_{231},\gamma_{312},\gamma_{321}\in M\otimes
W_{D^{3}}$. As long as the following three expressions are well defined, they
sum up only to vanish:
\begin{align*}
&  (\gamma_{123}\overset{\cdot}{\underset{1}{-}}\gamma_{132})\overset{\cdot
}{-}(\gamma_{231}\overset{\cdot}{\underset{1}{-}}\gamma_{321})\\
&  (\gamma_{231}\overset{\cdot}{\underset{2}{-}}\gamma_{213})\overset{\cdot
}{-}(\gamma_{312}\overset{\cdot}{\underset{2}{-}}\gamma_{132})\\
&  (\gamma_{312}\overset{\cdot}{\underset{3}{-}}\gamma_{321})\overset{\cdot
}{-}(\gamma_{123}\overset{\cdot}{\underset{3}{-}}\gamma_{213})
\end{align*}

\end{theorem}

Now we set out on a long journey to establish the above theorem. Let us begin with

\begin{proposition}
\label{t5.9}The diagram
\[%
\begin{array}
[c]{ccccc}
&  & \mathrm{id}_{M}\otimes W_{\eta_{1}^{1}} &  & \\
& M\otimes W_{E[1]} & \rightarrow & M\otimes W_{D^{4}\{(2,4),(3,4)\}} & \\
\mathrm{id}_{M}\otimes W_{\eta_{2}^{1}} & \downarrow &  & \downarrow &
\mathrm{id}_{M}\otimes W_{i_{14}^{1}}\\
& M\otimes W_{D^{4}\{(2,4),(3,4)\}} & \rightarrow & M\otimes W_{D(2)} & \\
&  & \mathrm{id}_{M}\otimes W_{i_{14}^{1}} &  &
\end{array}
\]
is a pullback, where the putative object $E[1]$ is
\begin{align*}
&  D^{7}\{(2,6),(3,6),(4,6),(5,6),(1,7),(2,7),(3,7),(4,7),(5,7),(6,7),(2,4),\\
&  (2,5),(3,4),(3,5)\}
\end{align*}
the putative mapping $i_{14}^{1}:D(2)\rightarrow D^{4}\{(2,4),(3,4)\}$ is
\[
(d_{1},d_{2})\in D(2)\mapsto(d_{1},0,0,d_{2})\in D^{4}\{(2,4),(3,4)\}
\]
the putative mapping $\eta_{1}^{1}:D^{4}\{(2,4),(3,4)\}\rightarrow E[1]$ is
\begin{align*}
(d_{1},d_{2},d_{3},d_{4})  &  \in D^{4}\{(2,4),(3,4)\}\mapsto\\
(d_{1},d_{2},d_{3},0,0,d_{4},0)  &  \in E[1]
\end{align*}
and the putative mapping $\eta_{2}^{1}:D^{4}\{(2,4),(3,4)\}\rightarrow E[1]$
is
\begin{align*}
(d_{1},d_{2},d_{3},d_{4})  &  \in D^{4}\{(2,4),(3,4)\}\mapsto\\
(d_{1},0,0,d_{2},d_{3},d_{4},d_{1}d_{4})  &  \in E[1]
\end{align*}

\end{proposition}

\begin{proof}
This follows from the microlinearity of $M$ and the pullback diagram of Weil
algebras
\[%
\begin{array}
[c]{ccccc}
&  & W_{\eta_{1}^{1}} &  & \\
& W_{E[1]} & \rightarrow & W_{D^{4}\{(2,4),(3,4)\}} & \\
W_{\eta_{2}^{1}} & \downarrow &  & \downarrow & W_{i_{14}^{1}}\\
& W_{D^{4}\{(2,4),(3,4)\}} & \rightarrow & W_{D(2)} & \\
&  & W_{i_{14}^{1}} &  &
\end{array}
\]

\end{proof}

\begin{notation}
We will write $\iota_{1}^{1}$, $\iota_{2}^{1}$, $\iota_{3}^{1}$ and $\iota
_{4}^{1}$ for the putative mappings $\eta_{1}^{1}\circ\varphi_{1}^{3}$,
$\eta_{1}^{1}\circ\psi_{1}^{3}$, $\eta_{2}^{1}\circ\varphi_{1}^{3}\ $and
$\eta_{2}^{1}\circ\psi_{1}^{3}$ respectively. That is to say, we have
\begin{align*}
\iota_{1}^{1}  &  :(d_{1},d_{2},d_{3})\in D^{3}\mapsto(d_{1},d_{2}%
,d_{3},0,0,0,0)\in E[1]\\
\iota_{2}^{1}  &  :(d_{1},d_{2},d_{3})\in D^{3}\mapsto(d_{1},d_{2}%
,d_{3},0,0,d_{2}d_{3},0)\in E[1]\\
\iota_{3}^{1}  &  :(d_{1},d_{2},d_{3})\in D^{3}\mapsto(d_{1},0,0,d_{2}%
,d_{3},0,0)\in E[1]\\
\iota_{4}^{1}  &  :(d_{1},d_{2},d_{3})\in D^{3}\mapsto(d_{1},0,0,d_{2}%
,d_{3},d_{2}d_{3},d_{1}d_{2}d_{3})\in E[1]
\end{align*}

\end{notation}

\begin{corollary}
For any $\gamma_{1},\gamma_{2},\gamma_{3},\gamma_{4}\in M\otimes W_{D^{3}}$,
if the expression
\[
(\gamma_{4}\overset{\cdot}{\underset{1}{-}}\gamma_{3})\overset{\cdot}%
{-}(\gamma_{2}\overset{\cdot}{\underset{1}{-}}\gamma_{1})
\]
is well defined, then there exists unique $\gamma\in M\otimes W_{E[1]}$ such
that $(\mathrm{id}_{M}\otimes W_{\iota_{i}^{1}})(\gamma)=\gamma_{i}$
($i=1,2,3,4$).
\end{corollary}

\begin{remark}
This means that $\gamma$ encodes $\gamma_{1}$, $\gamma_{2}$, $\gamma_{3}$ and
$\gamma_{4}$, which are in turn recovered from $\gamma$ via $\mathrm{id}%
_{M}\otimes W_{\iota_{i}^{1}}$'s.
\end{remark}

\begin{notation}
We will write $h_{(\gamma_{1},\gamma_{2},\gamma_{3},\gamma_{4})}^{1}$ for
$\gamma$ in the above corollary.
\end{notation}

\begin{remark}
We note that
\begin{align*}
g_{(\gamma_{1},\gamma_{2})}^{1}  &  =(\mathrm{id}_{M}\otimes W_{\eta_{1}^{1}%
})(h_{(\gamma_{1},\gamma_{2},\gamma_{3},\gamma_{4})}^{1})\\
g_{(\gamma_{3},\gamma_{4})}^{1}  &  =(\mathrm{id}_{M}\otimes W_{\eta_{2}^{1}%
})(h_{(\gamma_{1},\gamma_{2},\gamma_{3},\gamma_{4})}^{1})
\end{align*}
Therefore we have
\begin{align*}
\gamma_{2}\overset{\cdot}{\underset{1}{-}}\gamma_{1}  &  =(\mathrm{id}%
_{M}\otimes W_{(d_{1},d_{2})\in D^{2}\mapsto(d_{1},0,0,d_{2})\in
D^{4}\{(2,4),(3,4)\}})(g_{(\gamma_{1},\gamma_{2})}^{1})\\
&  =(\mathrm{id}_{M}\otimes W_{(d_{1},d_{2})\in D^{2}\mapsto(d_{1}%
,0,0,d_{2})\in D^{4}\{(2,4),(3,4)\}})\circ(\mathrm{id}_{M}\otimes W_{\eta
_{1}^{1}})(h_{(\gamma_{1},\gamma_{2},\gamma_{3},\gamma_{4})}^{1})\\
&  =(\mathrm{id}_{M}\otimes W_{(d_{1},d_{2})\in D^{2}\mapsto(d_{1}%
,0,0,0,0,d_{2},0)\in E[1]})(h_{(\gamma_{1},\gamma_{2},\gamma_{3},\gamma_{4}%
)}^{1})\\
\gamma_{4}\overset{\cdot}{\underset{1}{-}}\gamma_{3}  &  =(\mathrm{id}%
_{M}\otimes W_{(d_{1},d_{2})\in D^{2}\mapsto(d_{1},0,0,d_{2})\in
D^{4}\{(2,4),(3,4)\}})(g_{(\gamma_{3},\gamma_{4})}^{1})\\
&  =(\mathrm{id}_{M}\otimes W_{(d_{1},d_{2})\in D^{2}\mapsto(d_{1}%
,0,0,d_{2})\in D^{4}\{(2,4),(3,4)\}})\circ(\mathrm{id}_{M}\otimes W_{\eta
_{2}^{1}})(h_{(\gamma_{1},\gamma_{2},\gamma_{3},\gamma_{4})}^{1})\\
&  =(\mathrm{id}_{M}\otimes W_{(d_{1},d_{2})\in D^{2}\mapsto(d_{1}%
,0,0,0,0,d_{2},d_{1}d_{2})\in E[1]})(h_{(\gamma_{1},\gamma_{2},\gamma
_{3},\gamma_{4})}^{1})
\end{align*}
Thus we have
\[
g_{(\gamma_{2}\overset{\cdot}{\underset{1}{-}}\gamma_{1},\gamma_{4}%
\overset{\cdot}{\underset{1}{-}}\gamma_{3})}=(\mathrm{id}_{M}\otimes
W_{(d_{1},d_{2},d_{3})\in D^{3}\{(1,3),(2,3)\}\mapsto(d_{1},0,0,0,0,d_{2}%
,d_{3})\in E[1]})(h_{(\gamma_{1},\gamma_{2},\gamma_{3},\gamma_{4})}^{1})
\]
Finally we have
\begin{align*}
&  (\gamma_{4}\overset{\cdot}{\underset{1}{-}}\gamma_{3})\overset{\cdot}%
{-}(\gamma_{2}\overset{\cdot}{\underset{1}{-}}\gamma_{1})\\
&  =(\mathrm{id}_{M}\otimes W_{d\in D\mapsto(0,0,d)\in D^{3}\{(1,3),(2,3)\}}%
)(g_{(\gamma_{2}\overset{\cdot}{\underset{1}{-}}\gamma_{1},\gamma_{4}%
\overset{\cdot}{\underset{1}{-}}\gamma_{3})})\\
&  =(\mathrm{id}_{M}\otimes W_{d\in D\mapsto(0,0,d)\in D^{3}\{(1,3),(2,3)\}}%
)\circ(\mathrm{id}_{M}\otimes W_{(d_{1},d_{2},d_{3})\in D^{3}%
\{(1,3),(2,3)\}\mapsto(d_{1},0,0,0,0,d_{2},d_{3})\in E[1]})\\
&  (h_{(\gamma_{1},\gamma_{2},\gamma_{3},\gamma_{4})}^{1})\\
&  =(\mathrm{id}_{M}\otimes W_{d\in D\mapsto(0,0,0,0,0,0,d)\in E[1]}%
)(h_{(\gamma_{1},\gamma_{2},\gamma_{3},\gamma_{4})}^{1})
\end{align*}

\end{remark}

\begin{proposition}
\label{t5.10}The diagram
\[%
\begin{array}
[c]{ccccc}
&  & \mathrm{id}_{M}\otimes W_{\eta_{1}^{2}} &  & \\
& M\otimes W_{E[2]} & \rightarrow & M\otimes W_{D^{4}\{(1,4),(3,4)\}} & \\
\mathrm{id}_{M}\otimes W_{\eta_{2}^{2}} & \downarrow &  & \downarrow &
\mathrm{id}_{M}\otimes W_{i_{24}^{2}}\\
& M\otimes W_{D^{4}\{(1,4),(3,4)\}} & \rightarrow & M\otimes W_{D(2)} & \\
&  & \mathrm{id}_{M}\otimes W_{i_{24}^{2}} &  &
\end{array}
\]
is a pullback, where the putative object $E[2]$ is
\begin{align*}
&  D^{7}\{(1,6),(3,6),(4,6),(5,6),(1,7),(2,7),(3,7),(4,7),(5,7),(6,7),(1,4),\\
&  (1,5),(3,4),(3,5)\}
\end{align*}
the putative mapping $i_{24}^{2}:D(2)\rightarrow D^{4}\{(1,4),(3,4)\}$ is
\[
(d_{1},d_{2})\in D(2)\mapsto(0,d_{1},0,d_{2})\in D^{4}\{(1,4),(3,4)\}
\]
the putative mapping $\eta_{1}^{2}:D^{4}\{(1,4),(3,4)\}\}\rightarrow E[2]$ is
\begin{align*}
(d_{1},d_{2},d_{3},d_{4})  &  \in D^{4}\{(1,4),(3,4)\}\mapsto\\
(d_{1},d_{2},d_{3},0,0,d_{4},0)  &  \in E[2]
\end{align*}
and the putative mapping $\eta_{2}^{2}:D^{4}\{(1,4),(3,4)\}\rightarrow E[2]$
is
\begin{align*}
(d_{1},d_{2},d_{3},d_{4})  &  \in D^{4}\{(1,4),(3,4)\}\mapsto\\
(0,d_{2},0,d_{1},d_{3},d_{4},d_{2}d_{4})  &  \in E[2]
\end{align*}

\end{proposition}

\begin{proof}
This follows from the microlinearity of $M$ and the pullback diagram of Weil
algebras
\[%
\begin{array}
[c]{ccccc}
&  & W_{\eta_{1}^{2}} &  & \\
& W_{E[2]} & \rightarrow & W_{D^{4}\{(1,4),(3,4)\}} & \\
W_{\eta_{2}^{2}} & \downarrow &  & \downarrow & W_{i_{24}^{2}}\\
& W_{D^{4}\{(1,4),(3,4)\}} & \rightarrow & W_{D(2)} & \\
&  & W_{i_{24}^{2}} &  &
\end{array}
\]

\end{proof}

\begin{notation}
We will write $\iota_{1}^{2}$, $\iota_{2}^{2}$, $\iota_{3}^{2}$ and $\iota
_{4}^{2}$ for the putative mappings $\eta_{1}^{2}\circ\varphi_{2}^{3}$,
$\eta_{1}^{2}\circ\psi_{2}^{3}$, $\eta_{2}^{2}\circ\varphi_{2}^{3}\ $and
$\eta_{2}^{2}\circ\psi_{2}^{3}$ respectively. That is to say, we have
\begin{align*}
\iota_{1}^{2}  &  :(d_{1},d_{2},d_{3})\in D^{3}\mapsto(d_{1},d_{2}%
,d_{3},0,0,0,0)\in E[2]\\
\iota_{2}^{2}  &  :(d_{1},d_{2},d_{3})\in D^{3}\mapsto(d_{1},d_{2}%
,d_{3},0,0,d_{2}d_{3},0)\in E[2]\\
\iota_{3}^{2}  &  :(d_{1},d_{2},d_{3})\in D^{3}\mapsto(0,d_{2},0,d_{3}%
,d_{1},0,0)\in E[2]\\
\iota_{4}^{2}  &  :(d_{1},d_{2},d_{3})\in D^{3}\mapsto(0,d_{2},0,d_{3}%
,d_{1},d_{1}d_{3},d_{1}d_{2}d_{3})\in E[2]
\end{align*}

\end{notation}

\begin{corollary}
For any $\gamma_{1},\gamma_{2},\gamma_{3},\gamma_{4}\in M\otimes W_{D^{3}}$,
if the expression
\[
(\gamma_{4}\overset{\cdot}{\underset{2}{-}}\gamma_{3})\overset{\cdot}%
{-}(\gamma_{2}\overset{\cdot}{\underset{2}{-}}\gamma_{1})
\]
is well defined, then there exists unique $\gamma\in M\otimes W_{E[2]}$ such
that $(\mathrm{id}_{M}\otimes W_{\iota_{i}^{2}})(\gamma)=\gamma_{i}$
($i=1,2,3,4$).
\end{corollary}

\begin{remark}
This means that $\gamma$ encodes $\gamma_{1}$, $\gamma_{2}$, $\gamma_{3}$ and
$\gamma_{4}$, which are in turn recovered from $\gamma$ via $\mathrm{id}%
_{M}\otimes W_{\iota_{i}^{2}}$'s.
\end{remark}

\begin{notation}
We will write $h_{(\gamma_{1},\gamma_{2},\gamma_{3},\gamma_{4})}^{2}$ for
$\gamma$ in the above corollary.
\end{notation}

\begin{remark}
We note that
\begin{align*}
g_{(\gamma_{1},\gamma_{2})}^{2}  &  =(\mathrm{id}_{M}\otimes W_{\eta_{1}^{2}%
})(h_{(\gamma_{1},\gamma_{2},\gamma_{3},\gamma_{4})}^{2})\\
g_{(\gamma_{3},\gamma_{4})}^{2}  &  =(\mathrm{id}_{M}\otimes W_{\eta_{2}^{2}%
})(h_{(\gamma_{1},\gamma_{2},\gamma_{3},\gamma_{4})}^{2})
\end{align*}
Therefore we have
\begin{align*}
\gamma_{2}\overset{\cdot}{\underset{2}{-}}\gamma_{1}  &  =(\mathrm{id}%
_{M}\otimes W_{(d_{1},d_{2})\in D^{2}\mapsto(0,d_{1},0,d_{2})\in
D^{4}\{(1,4),(3,4)\}})(g_{(\gamma_{1},\gamma_{2})}^{2})\\
&  =(\mathrm{id}_{M}\otimes W_{(d_{1},d_{2})\in D^{2}\mapsto(0,d_{1}%
,0,d_{2})\in D^{4}\{(1,4),(3,4)\}})\circ(\mathrm{id}_{M}\otimes W_{\eta
_{1}^{2}})(h_{(\gamma_{1},\gamma_{2},\gamma_{3},\gamma_{4})}^{2})\\
&  =(\mathrm{id}_{M}\otimes W_{(d_{1},d_{2})\in D^{2}\mapsto(0,d_{1}%
,0,0,0,d_{2},0)\in E[2]})(h_{(\gamma_{1},\gamma_{2},\gamma_{3},\gamma_{4}%
)}^{2})\\
\gamma_{4}\overset{\cdot}{\underset{2}{-}}\gamma_{3}  &  =(\mathrm{id}%
_{M}\otimes W_{(d_{1},d_{2})\in D^{2}\mapsto(0,d_{1},0,d_{2})\in
D^{4}\{(1,4),(3,4)\}})(g_{(\gamma_{3},\gamma_{4})}^{2})\\
&  =(\mathrm{id}_{M}\otimes W_{(d_{1},d_{2})\in D^{2}\mapsto(0,d_{1}%
,0,d_{2})\in D^{4}\{(1,4),(3,4)\}})\circ(\mathrm{id}_{M}\otimes W_{\eta
_{2}^{2}})(h_{(\gamma_{1},\gamma_{2},\gamma_{3},\gamma_{4})}^{2})\\
&  =(\mathrm{id}_{M}\otimes W_{(d_{1},d_{2})\in D^{2}\mapsto(0,d_{1}%
,0,0,0,d_{2},d_{1}d_{2})\in E[2]})(h_{(\gamma_{1},\gamma_{2},\gamma_{3}%
,\gamma_{4})}^{2})
\end{align*}
Thus we have
\[
g_{(\gamma_{2}\overset{\cdot}{\underset{2}{-}}\gamma_{1},\gamma_{4}%
\overset{\cdot}{\underset{2}{-}}\gamma_{3})}=(\mathrm{id}_{M}\otimes
W_{(d_{1},d_{2},d_{3})\in D^{3}\{(1,3),(2,3)\}\mapsto(0,d_{1},0,0,0,d_{2}%
,d_{3})\in E[2]})(h_{(\gamma_{1},\gamma_{2},\gamma_{3},\gamma_{4})}^{2})
\]
Finally we have
\begin{align*}
&  (\gamma_{4}\overset{\cdot}{\underset{2}{-}}\gamma_{3})\overset{\cdot}%
{-}(\gamma_{2}\overset{\cdot}{\underset{2}{-}}\gamma_{1})\\
&  =(\mathrm{id}_{M}\otimes W_{d\in D\mapsto(0,0,d)\in D^{3}\{(1,3),(2,3)\}}%
)(g_{(\gamma_{2}\overset{\cdot}{\underset{2}{-}}\gamma_{1},\gamma_{4}%
\overset{\cdot}{\underset{2}{-}}\gamma_{3})})\\
&  =(\mathrm{id}_{M}\otimes W_{d\in D\mapsto(0,0,d)\in D^{3}\{(1,3),(2,3)\}}%
)\circ(\mathrm{id}_{M}\otimes W_{(d_{1},d_{2},d_{3})\in D^{3}%
\{(1,3),(2,3)\}\mapsto(0,d_{1},0,0,0,d_{2},d_{3})\in E[2]})\\
&  (h_{(\gamma_{1},\gamma_{2},\gamma_{3},\gamma_{4})}^{2})\\
&  =(\mathrm{id}_{M}\otimes W_{d\in D\mapsto(0,0,0,0,0,0,d)\in E[2]}%
)(h_{(\gamma_{1},\gamma_{2},\gamma_{3},\gamma_{4})}^{2})
\end{align*}

\end{remark}

\begin{proposition}
\label{t5.11}The diagram
\[%
\begin{array}
[c]{ccccc}
&  & \mathrm{id}_{M}\otimes W_{\eta_{1}^{3}} &  & \\
& M\otimes W_{E[3]} & \rightarrow & M\otimes W_{D^{4}\{(1,4),(2,4)\}} & \\
\mathrm{id}_{M}\otimes W_{\eta_{2}^{3}} & \downarrow &  & \downarrow &
\mathrm{id}_{M}\otimes W_{i_{34}^{3}}\\
& M\otimes W_{D^{4}\{(1,4),(2,4)\}} & \rightarrow & M\otimes W_{D(2)} & \\
&  & \mathrm{id}_{M}\otimes W_{i_{34}^{3}} &  &
\end{array}
\]
is a pullback, where the putative object $E[3]$ is
\begin{align*}
&  D^{7}\{(1,6),(2,6),(4,6),(5,6),(1,7),(2,7),(3,7),(4,7),(5,7),(6,7),(1,4),\\
&  (1,5),(2,4),(2,5)\}\}
\end{align*}
the putative mapping $i_{34}^{3}:D(2)\rightarrow D^{4}\{(1,4),(2,4)\}$ is
\[
(d_{1},d_{2})\in D(2)\mapsto(0,0,d_{1},d_{2})\in D^{4}\{(1,4),(2,4)\}
\]
the putative mapping $\eta_{1}^{3}:D^{4}\{(1,4),(2,4)\}\}\rightarrow E[3]$ is
\begin{align*}
(d_{1},d_{2},d_{3},d_{4})  &  \in D^{4}\{(1,4),(2,4)\}\mapsto\\
(d_{1},d_{2},d_{3},0,0,d_{4},0)  &  \in E[3]
\end{align*}
and the putative mapping $\eta_{2}^{3}:D^{4}\{(1,4),(3,4)\}\rightarrow E[3]$
is
\begin{align*}
(d_{1},d_{2},d_{3},d_{4})  &  \in D^{4}\{(1,4),(2,4)\}\mapsto\\
(0,0,d_{3},d_{1},d_{2},d_{4},d_{3}d_{4})  &  \in E[3]
\end{align*}

\end{proposition}

\begin{proof}
This follows from the microlinearity of $M$ and the pullback diagram of Weil
algebras
\[%
\begin{array}
[c]{ccccc}
&  & W_{\eta_{1}^{3}} &  & \\
& W_{E[3]} & \rightarrow & W_{D^{4}\{(1,4),(2,4)\}} & \\
W_{\eta_{2}^{3}} & \downarrow &  & \downarrow & W_{i_{34}^{3}}\\
& W_{D^{4}\{(1,4),(2,4)\}} & \rightarrow & W_{D(2)} & \\
&  & W_{i_{34}^{3}} &  &
\end{array}
\]

\end{proof}

\begin{notation}
We will write $\iota_{1}^{3}$, $\iota_{2}^{3}$, $\iota_{3}^{3}$ and $\iota
_{4}^{3}$ for the putative mappings $\eta_{1}^{3}\circ\varphi_{3}^{3}$,
$\eta_{1}^{3}\circ\psi_{3}^{3}$, $\eta_{2}^{3}\circ\varphi_{3}^{3}\ $and
$\eta_{2}^{3}\circ\psi_{3}^{3}$ respectively. That is to say, we have
\begin{align*}
\iota_{1}^{3}  &  :(d_{1},d_{2},d_{3})\in D^{3}\mapsto(d_{1},d_{2}%
,d_{3},0,0,0,0)\in E[3]\\
\iota_{2}^{3}  &  :(d_{1},d_{2},d_{3})\in D^{3}\mapsto(d_{1},d_{2}%
,d_{3},0,0,d_{1}d_{2},0)\in E[3]\\
\iota_{3}^{3}  &  :(d_{1},d_{2},d_{3})\in D^{3}\mapsto(0,0,d_{3},d_{1}%
,d_{2},0,0)\in E[3]\\
\iota_{4}^{3}  &  :(d_{1},d_{2},d_{3})\in D^{3}\mapsto(0,0,d_{3},d_{1}%
,d_{2},d_{1}d_{2},d_{1}d_{2}d_{3})\in E[3]
\end{align*}

\end{notation}

\begin{corollary}
For any $\gamma_{1},\gamma_{2},\gamma_{3},\gamma_{4}\in M\otimes W_{D^{3}}$,
if the expression
\[
(\gamma_{4}\overset{\cdot}{\underset{3}{-}}\gamma_{3})\overset{\cdot}%
{-}(\gamma_{2}\overset{\cdot}{\underset{3}{-}}\gamma_{1})
\]
is well defined, then there exists unique $\gamma\in M\otimes W_{E[3]}$ such
that $(\mathrm{id}_{M}\otimes W_{\iota_{i}^{3}})(\gamma)=\gamma_{i}$
($i=1,2,3,4$).
\end{corollary}

\begin{remark}
This means that $\gamma$ encodes $\gamma_{1}$, $\gamma_{2}$, $\gamma_{3}$ and
$\gamma_{4}$, which are in turn recovered from $\gamma$ via $\mathrm{id}%
_{M}\otimes W_{\iota_{i}^{3}}$'s.
\end{remark}

\begin{notation}
We will write $h_{(\gamma_{1},\gamma_{2},\gamma_{3},\gamma_{4})}^{3}$ for
$\gamma$ in the above corollary.
\end{notation}

\begin{remark}
We note that
\begin{align*}
g_{(\gamma_{1},\gamma_{2})}^{3}  &  =(\mathrm{id}_{M}\otimes W_{\eta_{1}^{3}%
})(h_{(\gamma_{1},\gamma_{2},\gamma_{3},\gamma_{4})}^{3})\\
g_{(\gamma_{3},\gamma_{4})}^{3}  &  =(\mathrm{id}_{M}\otimes W_{\eta_{2}^{3}%
})(h_{(\gamma_{1},\gamma_{2},\gamma_{3},\gamma_{4})}^{3})
\end{align*}
Therefore we have
\begin{align*}
\gamma_{2}\overset{\cdot}{\underset{3}{-}}\gamma_{1}  &  =(\mathrm{id}%
_{M}\otimes W_{(d_{1},d_{2})\in D^{2}\mapsto(0,0,d_{1},d_{2})\in
D^{4}\{(1,4),(2,4)\}})(g_{(\gamma_{1},\gamma_{2})}^{3})\\
&  =(\mathrm{id}_{M}\otimes W_{(d_{1},d_{2})\in D^{2}\mapsto(0,0,d_{1}%
,d_{2})\in D^{4}\{(1,4),(2,4)\}})\circ(\mathrm{id}_{M}\otimes W_{\eta_{1}^{3}%
})(h_{(\gamma_{1},\gamma_{2},\gamma_{3},\gamma_{4})}^{3})\\
&  =(\mathrm{id}_{M}\otimes W_{(d_{1},d_{2})\in D^{2}\mapsto(0,0,d_{1}%
,0,0,d_{2},0)\in E[3]})(h_{(\gamma_{1},\gamma_{2},\gamma_{3},\gamma_{4})}%
^{3})\\
\gamma_{4}\overset{\cdot}{\underset{3}{-}}\gamma_{3}  &  =(\mathrm{id}%
_{M}\otimes W_{(d_{1},d_{2})\in D^{2}\mapsto(0,0,d_{1},d_{2})\in
D^{4}\{(1,4),(2,4)\}})(g_{(\gamma_{3},\gamma_{4})}^{3})\\
&  =(\mathrm{id}_{M}\otimes W_{(d_{1},d_{2})\in D^{2}\mapsto(0,0,d_{1}%
,d_{2})\in D^{4}\{(1,4),(2,4)\}})\circ(\mathrm{id}_{M}\otimes W_{\eta_{2}^{3}%
})(h_{(\gamma_{1},\gamma_{2},\gamma_{3},\gamma_{4})}^{3})\\
&  =(\mathrm{id}_{M}\otimes W_{(d_{1},d_{2})\in D^{2}\mapsto(0,0,d_{1}%
,0,0,d_{2},d_{1}d_{2})\in E[3]})(h_{(\gamma_{1},\gamma_{2},\gamma_{3}%
,\gamma_{4})}^{3})
\end{align*}
Thus we have
\[
g_{(\gamma_{2}\overset{\cdot}{\underset{3}{-}}\gamma_{1},\gamma_{4}%
\overset{\cdot}{\underset{3}{-}}\gamma_{3})}=(\mathrm{id}_{M}\otimes
W_{(d_{1},d_{2},d_{3})\in D^{3}\{(1,3),(2,3)\}\mapsto(0,0,d_{1},0,0,d_{2}%
,d_{3})\in E[3]})(h_{(\gamma_{1},\gamma_{2},\gamma_{3},\gamma_{4})}^{3})
\]
Finally we have
\begin{align*}
&  (\gamma_{4}\overset{\cdot}{\underset{3}{-}}\gamma_{3})\overset{\cdot}%
{-}(\gamma_{2}\overset{\cdot}{\underset{3}{-}}\gamma_{1})\\
&  =(\mathrm{id}_{M}\otimes W_{d\in D\mapsto(0,0,d)\in D^{3}\{(1,3),(2,3)\}}%
)(g_{(\gamma_{2}\overset{\cdot}{\underset{3}{-}}\gamma_{1},\gamma_{4}%
\overset{\cdot}{\underset{3}{-}}\gamma_{3})})\\
&  =(\mathrm{id}_{M}\otimes W_{d\in D\mapsto(0,0,d)\in D^{3}\{(1,3),(2,3)\}}%
)\circ(\mathrm{id}_{M}\otimes W_{(d_{1},d_{2},d_{3})\in D^{3}%
\{(1,3),(2,3)\}\mapsto(0,0,d_{1},0,0,d_{2},d_{3})\in E[3]})\\
&  (h_{(\gamma_{1},\gamma_{2},\gamma_{3},\gamma_{4})}^{3})\\
&  =(\mathrm{id}_{M}\otimes W_{d\in D\mapsto(0,0,0,0,0,0,d)\in E[3]}%
)(h_{(\gamma_{1},\gamma_{2},\gamma_{3},\gamma_{4})}^{3})
\end{align*}

\end{remark}

Now we come to the crucial step in the proof of the general Jacobi identity.

\begin{theorem}
\label{t5.12}The diagram
\[%
\begin{array}
[c]{ccccccc}
& \mathrm{id}_{M}\otimes W_{h_{12}^{1}} &  & M\otimes W_{E[1]} &  &
\mathrm{id}_{M}\otimes W_{h_{31}^{1}} & \\
&  & \swarrow & \uparrow & \searrow &  & \\
& M\otimes W_{D^{3}\oplus D^{3}} &  & M\otimes W_{G} &  & M\otimes
W_{D^{3}\oplus D^{3}} & \\
\mathrm{id}_{M}\otimes W_{h_{12}^{2}} & \uparrow & \swarrow &  & \searrow &
\uparrow & \mathrm{id}_{M}\otimes W_{h_{31}^{3}}\\
& M\otimes W_{E[2]} &  &  &  & M\otimes W_{E[3]} & \\
&  & \searrow &  & \swarrow &  & \\
& \mathrm{id}_{M}\otimes W_{h_{23}^{2}} &  & M\otimes W_{D^{3}\oplus D^{3}} &
& \mathrm{id}_{M}\otimes W_{h_{23}^{3}} &
\end{array}
\]
is a limit diagram with the three unnamed arrows being
\begin{align*}
\mathrm{id}_{M}\otimes W_{k_{1}}  &  :M\otimes W_{G}\rightarrow M\otimes
W_{E[1]}\\
\mathrm{id}_{M}\otimes W_{k_{2}}  &  :M\otimes W_{G}\rightarrow M\otimes
W_{E[2]}\\
\mathrm{id}_{M}\otimes W_{k_{3}}  &  :M\otimes W_{G}\rightarrow M\otimes
W_{E[3]}%
\end{align*}
where the putative object $G$ is
\begin{align*}
&  D^{8}%
\{(2,4),(3,4),(1,5),(3,5),(1,6),(2,6),(4,5),(4,6),(5,6),(1,7),(2,7),(3,7),\\
&  (4,7),(5,7),(6,7),(1,8),(2,8),(3,8),(4,8),(5,8),(6,8),(7,8)\}\text{,}%
\end{align*}
the putative mapping $k_{1}:E[1]\rightarrow G$ is
\begin{align*}
(d_{1},d_{2},d_{3},d_{4},d_{5},d_{6},d_{7})  &  \in E[1]\mapsto\\
(d_{1},d_{2}+d_{4},d_{3}+d_{5},d_{6}-d_{2}d_{3}-d_{4}d_{5},-d_{1}d_{5}%
,d_{1}d_{4},d_{7}+d_{1}d_{2}d_{3},d_{1}d_{2}d_{3})  &  \in G\text{,}%
\end{align*}
the putative mapping $k_{2}:E[2]\rightarrow G$ is
\begin{align*}
(d_{1},d_{2},d_{3},d_{4},d_{5},d_{6},d_{7})  &  \in E[2]\mapsto\\
(d_{1}+d_{5},d_{2},d_{3}+d_{4},-d_{2}d_{3},d_{6}-d_{1}d_{3}-d_{4}d_{5}%
,d_{1}d_{2},d_{2}d_{4}d_{5},d_{7})  &  \in G\text{,}%
\end{align*}
the putative mapping $k_{3}:E[3]\rightarrow G$ is
\begin{align*}
(d_{1},d_{2},d_{3},d_{4},d_{5},d_{6},d_{7})  &  \in E[3]\mapsto\\
(d_{1}+d_{4},d_{2}+d_{5},d_{3},-d_{4}d_{5},-d_{1}d_{3},d_{6},-d_{7}%
,-d_{7}+d_{3}d_{4}d_{5})  &  \in G\text{,}%
\end{align*}
the putative mapping $h_{12}^{1}$ is $\iota_{2}^{1}\oplus\iota_{3}^{1}$, the
putative mapping $h_{12}^{2}$ is $\iota_{4}^{2}\oplus\iota_{1}^{2}$, the
putative mapping $h_{23}^{2}$ is $\iota_{2}^{2}\oplus\iota_{3}^{2}$, the
putative mapping $h_{23}^{3}$ is $\iota_{4}^{3}\oplus\iota_{1}^{3}$, the
putative mapping $h_{31}^{3}$ is $\iota_{2}^{3}\oplus\iota_{3}^{3}$, and the
putative mapping $h_{31}^{1}$ is $\iota_{4}^{1}\oplus\iota_{1}^{2} $.
\end{theorem}

The proof of the above theorem follows directly from the following lemma.

\begin{lemma}
\label{t5.13}The following diagram is a limit diagram of Weil algebras:
\[%
\begin{array}
[c]{ccccccc}
& W_{h_{12}^{1}} &  & W_{E[1]} &  & W_{h_{31}^{1}} & \\
&  & \swarrow & \uparrow & \searrow &  & \\
& W_{D^{3}\oplus D^{3}} &  & W_{G} &  & W_{D^{3}\oplus D^{3}} & \\
W_{h_{12}^{2}} & \uparrow & \swarrow &  & \searrow & \uparrow & W_{h_{31}^{3}%
}\\
& W_{E[2]} &  &  &  & W_{E[3]} & \\
&  & \searrow &  & \swarrow &  & \\
& W_{h_{23}^{2}} &  & W_{D^{3}\oplus D^{3}} &  & W_{h_{23}^{3}} &
\end{array}
\]

\end{lemma}

\begin{proof}
Let $\gamma_{1}\in W_{E[1]}$, $\gamma_{2}\in W_{E[2]}$, $\gamma_{3}\in
W_{E[3]}$ and $\gamma\in W_{G}$ so that they are polynomials with real
coefficients in the following forms:
\begin{align*}
&  \gamma_{1}(X_{1},X_{2},X_{3},X_{4},X_{5},X_{6},X_{7})\\
&  =a^{1}+a_{1}^{1}X_{1}+a_{2}^{1}X_{2}+a_{3}^{1}X_{3}+a_{4}^{1}X_{4}%
+a_{5}^{1}X_{5}+a_{6}^{1}X_{6}+a_{7}^{1}X_{7}+a_{12}^{1}X_{1}X_{2}+a_{13}%
^{1}X_{1}X_{3}+\\
&  a_{14}^{1}X_{1}X_{4}+a_{15}^{1}X_{1}X_{5}+a_{16}^{1}X_{1}X_{6}+a_{23}%
^{1}X_{2}X_{3}+a_{45}^{1}X_{4}X_{5}+a_{123}^{1}X_{1}X_{2}X_{3}+a_{145}%
^{1}X_{1}X_{4}X_{5}%
\end{align*}
\begin{align*}
&  \gamma_{2}(X_{1},X_{2},X_{3},X_{4},X_{5},X_{6},X_{7})\\
&  =a^{2}+a_{1}^{2}X_{1}+a_{2}^{2}X_{2}+a_{3}^{2}X_{3}+a_{4}^{2}X_{4}%
+a_{5}^{2}X_{5}+a_{6}^{2}X_{6}+a_{7}^{2}X_{7}+a_{12}^{2}X_{1}X_{2}+a_{13}%
^{2}X_{1}X_{3}+\\
&  a_{23}^{2}X_{2}X_{3}+a_{24}^{2}X_{2}X_{4}+a_{25}^{2}X_{2}X_{5}+a_{26}%
^{2}X_{2}X_{6}+a_{45}^{2}X_{4}X_{5}+a_{123}^{2}X_{1}X_{2}X_{3}+a_{245}%
^{2}X_{2}X_{4}X_{5}%
\end{align*}

\begin{align*}
&  \gamma_{3}(X_{1},X_{2},X_{3},X_{4},X_{5},X_{6},X_{7})\\
&  =a^{3}+a_{1}^{3}X_{1}+a_{2}^{3}X_{2}+a_{3}^{3}X_{3}+a_{4}^{3}X_{4}%
+a_{5}^{3}X_{5}+a_{6}^{3}X_{6}+a_{7}^{3}X_{7}+a_{12}^{3}X_{1}X_{2}+a_{13}%
^{3}X_{1}X_{3}+\\
&  a_{23}^{3}X_{2}X_{3}+a_{34}^{3}X_{3}X_{4}+a_{35}^{3}X_{3}X_{5}+a_{36}%
^{3}X_{3}X_{6}+a_{45}^{3}X_{4}X_{5}+a_{123}^{3}X_{1}X_{2}X_{3}+a_{345}%
^{3}X_{3}X_{4}X_{5}%
\end{align*}
\begin{align*}
&  \gamma(X_{1},X_{2},X_{3},X_{4},X_{5},X_{6},X_{7},X_{8})\\
&  =b+b_{1}X_{1}+b_{2}X_{2}+b_{3}X_{3}+b_{4}X_{4}+b_{5}X_{5}+b_{6}X_{6}%
+b_{7}X_{7}+b_{8}X_{8}+b_{12}X_{1}X_{2}\\
&  +b_{13}X_{1}X_{3}+b_{14}X_{1}X_{4}+b_{23}X_{2}X_{3}+b_{25}X_{2}X_{5}%
+b_{36}X_{3}X_{6}%
\end{align*}
It is easy to see that
\begin{align*}
&  W_{h_{12}^{1}}(\gamma_{1})(X_{1},X_{2},X_{3},X_{4},X_{5},X_{6})\\
&  =a^{1}+a_{1}^{1}X_{1}+a_{2}^{1}X_{2}+a_{3}^{1}X_{3}+a_{6}^{1}X_{2}%
X_{3}+a_{12}^{1}X_{1}X_{2}+a_{13}^{1}X_{1}X_{3}+a_{16}^{1}X_{1}X_{2}%
X_{3}+a_{23}^{1}X_{2}X_{3}\\
&  +a_{123}^{1}X_{1}X_{2}X_{3}+a_{1}^{1}X_{4}+a_{4}^{1}X_{5}+a_{5}^{1}%
X_{6}+a_{14}^{1}X_{4}X_{5}+a_{15}^{1}X_{4}X_{6}+a_{45}^{1}X_{5}X_{6}\\
&  +a_{145}^{1}X_{4}X_{5}X_{6}\\
&  =a^{1}+a_{1}^{1}X_{1}+a_{2}^{1}X_{2}+a_{3}^{1}X_{3}+a_{12}^{1}X_{1}%
X_{2}+a_{13}^{1}X_{1}X_{3}+(a_{6}^{1}+a_{23}^{1})X_{2}X_{3}\\
&  +(a_{16}^{1}+a_{123}^{1})X_{1}X_{2}X_{3}+a_{1}^{1}X_{4}+a_{4}^{1}%
X_{5}+a_{5}^{1}X_{6}+a_{14}^{1}X_{4}X_{5}+a_{15}^{1}X_{4}X_{6}+a_{45}^{1}%
X_{5}X_{6}\\
&  +a_{145}^{1}X_{4}X_{5}X_{6}%
\end{align*}
\begin{align*}
&  W_{h_{12}^{2}}(\gamma_{2})(X_{1},X_{2},X_{3},X_{4},X_{5},X_{6})\\
&  =a^{2}+a_{2}^{2}X_{2}+a_{4}^{2}X_{3}+a_{5}^{2}X_{1}+a_{6}^{2}X_{1}%
X_{3}+a_{7}^{2}X_{1}X_{2}X_{3}+a_{24}^{2}X_{2}X_{3}+a_{25}^{2}X_{1}X_{2}\\
&  +a_{26}^{2}X_{1}X_{2}X_{3}+a_{45}^{2}X_{1}X_{3}+a_{245}^{2}X_{1}X_{2}%
X_{3}+a_{1}^{2}X_{4}+a_{2}^{2}X_{5}+a_{3}^{2}X_{6}+a_{12}^{2}X_{4}X_{5}\\
&  +a_{13}^{2}X_{4}X_{6}+a_{23}^{2}X_{5}X_{6}+a_{123}^{2}X_{4}X_{5}X_{6}\\
&  =a^{2}+a_{5}^{2}X_{1}+a_{2}^{2}X_{2}+a_{4}^{2}X_{3}+a_{25}^{2}X_{1}%
X_{2}+(a_{6}^{2}+a_{45}^{2})X_{1}X_{3}+a_{24}^{2}X_{2}X_{3}+\\
&  (a_{7}^{2}+a_{26}^{2}+a_{245}^{2})X_{1}X_{2}X_{3}+a_{1}^{2}X_{4}+a_{2}%
^{2}X_{5}+a_{3}^{2}X_{6}+a_{12}^{2}X_{4}X_{5}+a_{13}^{2}X_{4}X_{6}+a_{23}%
^{2}X_{5}X_{6}\\
&  +a_{123}^{2}X_{4}X_{5}X_{6}%
\end{align*}
Therefore the condition that $W_{h_{12}^{1}}(\gamma_{1})=W_{h_{12}^{2}}%
(\gamma_{2})$ is equivalent to the following conditions as a whole:
\begin{align}
a^{1}  &  =a^{2}\label{5.1}\\
a_{1}^{1}  &  =a_{5}^{2},\;a_{2}^{1}=a_{2}^{2},\;a_{3}^{1}=a_{4}^{2}%
,\;a_{1}^{1}=a_{1}^{2},\;a_{4}^{1}=a_{2}^{2},\;a_{5}^{1}=a_{3}^{2}%
\label{5.2}\\
a_{12}^{1}  &  =a_{25}^{2},\;a_{13}^{1}=a_{6}^{2}+a_{45}^{2},\;a_{6}%
^{1}+a_{23}^{1}=a_{24}^{2},\;a_{14}^{1}=a_{12}^{2},\;a_{15}^{1}=a_{13}%
^{2},\;\nonumber\\
a_{45}^{1}  &  =a_{23}^{2}\label{5.3}\\
a_{16}^{1}+a_{123}^{1}  &  =a_{7}^{2}+a_{26}^{2}+a_{245}^{2},\;a_{145}%
^{1}=a_{123}^{2}\label{5.4}%
\end{align}
By the same token, the condition that $W_{h_{23}^{2}}(\gamma_{2}%
)=W_{h_{23}^{3}}(\gamma_{3})$ is equivalent to the following conditions as a
whole:
\begin{align}
a^{2}  &  =a^{3}\label{5.5}\\
a_{2}^{2}  &  =a_{5}^{3},\;a_{3}^{2}=a_{3}^{3},\;a_{1}^{2}=a_{4}^{3}%
,\;a_{2}^{2}=a_{2}^{3},\;a_{4}^{2}=a_{3}^{3},\;a_{5}^{2}=a_{1}^{3}%
\label{5.6}\\
a_{23}^{2}  &  =a_{35}^{3},\;a_{12}^{2}=a_{6}^{3}+a_{45}^{3},\;a_{6}%
^{2}+a_{13}^{2}=a_{34}^{3},\;a_{24}^{2}=a_{23}^{3},\;a_{25}^{2}=a_{12}%
^{3},\;\nonumber\\
a_{45}^{2}  &  =a_{13}^{3}\label{5.7}\\
a_{26}^{2}+a_{123}^{2}  &  =a_{7}^{3}+a_{36}^{3}+a_{345}^{3},\;a_{245}%
^{2}=a_{123}^{3}\label{5.8}%
\end{align}
By the same token again, the condition that $W_{h_{31}^{3}}(\gamma
_{3})=W_{h_{31}^{1}}(\gamma_{1})$ is equivalent to the following conditions as
a whole:
\begin{align}
a^{3}  &  =a^{1}\label{5.9}\\
a_{3}^{3}  &  =a_{5}^{1},\;a_{1}^{3}=a_{1}^{1},\;a_{2}^{3}=a_{4}^{1}%
,\;a_{3}^{3}=a_{3}^{1},\;a_{4}^{3}=a_{1}^{1},\;a_{5}^{3}=a_{2}^{1}%
\label{5.10}\\
a_{13}^{3}  &  =a_{15}^{1},\;a_{23}^{3}=a_{6}^{1}+a_{45}^{1},\;a_{6}%
^{3}+a_{12}^{3}=a_{14}^{1},\;a_{34}^{3}=a_{13}^{1},\;a_{35}^{3}=a_{23}%
^{1},\;\nonumber\\
a_{45}^{3}  &  =a_{12}^{1}\label{5.11}\\
a_{36}^{3}+a_{123}^{3}  &  =a_{7}^{1}+a_{16}^{1}+a_{145}^{1},\;a_{345}%
^{3}=a_{123}^{1}\label{5.12}%
\end{align}
The three conditions (\ref{5.1}), (\ref{5.5}) and (\ref{5.9}) can be combined
into
\begin{equation}
a^{1}=a^{2}=a^{3}\label{5.13}%
\end{equation}
The three conditions (\ref{5.2}), (\ref{5.6}) and (\ref{5.10}) are to be
superseded by the following three conditions as a whole:
\begin{align}
a_{1}^{1}  &  =a_{1}^{2}=a_{1}^{3}=a_{5}^{2}=a_{4}^{3}\label{5.14}\\
a_{2}^{1}  &  =a_{2}^{2}=a_{2}^{3}=a_{4}^{1}=a_{5}^{3}\label{5.15}\\
a_{3}^{1}  &  =a_{3}^{2}=a_{3}^{3}=a_{5}^{1}=a_{4}^{2}\label{5.16}%
\end{align}
The three conditions (\ref{5.3}), (\ref{5.7}) and (\ref{5.11}) are equivalent
to the following six conditions as a whole:
\begin{align}
a_{12}^{1}  &  =a_{12}^{2}=a_{12}^{3}\label{5.17}\\
a_{13}^{1}  &  =a_{13}^{2}=a_{13}^{3}\label{5.18}\\
a_{23}^{1}  &  =a_{23}^{2}=a_{23}^{3}\label{5.19}\\
a_{14}^{1}  &  =a_{12}^{1}+a_{6}^{3},\;a_{15}^{1}=a_{13}^{1}-a_{6}%
^{2},\;a_{45}^{1}=a_{23}^{1}\label{5.20}\\
a_{24}^{2}  &  =a_{23}^{2}+a_{6}^{1},\;a_{25}^{2}=a_{12}^{2}-a_{6}%
^{3},\;a_{45}^{2}=a_{13}^{2}\label{5.21}\\
a_{34}^{3}  &  =a_{13}^{3}+a_{6}^{2},\;a_{35}^{3}=a_{23}^{3}-a_{6}%
^{1},\;a_{45}^{3}=a_{12}^{3}\label{5.22}%
\end{align}
The conditions (\ref{5.4}), (\ref{5.8}) and (\ref{5.12}) imply that
\begin{align}
&  a_{7}^{1}+a_{7}^{2}+a_{7}^{3}\nonumber\\
&  =(a_{36}^{3}+a_{123}^{3}-a_{16}^{1}-a_{145}^{1})+(a_{16}^{1}+a_{123}%
^{1}-a_{26}^{2}-a_{245}^{2})+(a_{26}^{2}+a_{123}^{2}-a_{36}^{3}-a_{345}%
^{3})\nonumber\\
&  =(a_{36}^{3}+a_{123}^{3}-a_{16}^{1}-a_{123}^{2})+(a_{16}^{1}+a_{123}%
^{1}-a_{26}^{2}-a_{123}^{3})+(a_{26}^{2}+a_{123}^{2}-a_{36}^{3}-a_{123}%
^{1})\nonumber\\
&  =0\label{5.23}%
\end{align}
Therefore the three conditions (\ref{5.4}), (\ref{5.8}) and (\ref{5.12}) are
to be replaced by the following five conditions as a whole:
\begin{align}
a_{145}^{1}-a_{123}^{1}  &  =a_{7}^{3}+a_{36}^{3}-a_{26}^{2}\label{5.24}\\
a_{245}^{2}-a_{123}^{2}  &  =a_{7}^{1}+a_{16}^{1}-a_{36}^{3}\label{5.25}\\
a_{345}^{3}-a_{123}^{3}  &  =a_{7}^{2}+a_{26}^{2}-a_{16}^{1}\label{5.26}\\
a_{145}^{1}  &  =a_{123}^{2},\;a_{245}^{2}=a_{123}^{3}\label{5.27}\\
a_{7}^{1}+a_{7}^{2}+a_{7}^{3}  &  =0\label{5.28}%
\end{align}
Indeed, the condition that $a_{345}^{3}=a_{123}^{1}$ is derivable from the
above five conditions, as is to be demonstrated in the following:
\begin{align*}
&  a_{345}^{3}\\
&  =a_{123}^{3}+a_{7}^{2}+a_{26}^{2}-a_{16}^{1}\text{ \ \ [(\ref{5.26})]}\\
&  =a_{245}^{2}+a_{7}^{2}+a_{26}^{2}-a_{16}^{1}\text{ \ \ [(\ref{5.27})]}\\
&  =a_{123}^{2}+a_{7}^{1}-a_{36}^{3}+a_{7}^{2}+a_{26}^{2}\text{
\ \ [(\ref{5.25})]}\\
&  =a_{145}^{1}+a_{7}^{1}-a_{36}^{3}+a_{7}^{2}+a_{26}^{2}\text{
\ \ [(\ref{5.27})]}\\
&  =a_{123}^{1}+a_{7}^{1}+a_{7}^{2}+a_{7}^{3}\text{ \ \ [(\ref{5.24})]}\\
&  =a_{123}^{1}\text{ \ \ [(\ref{5.28})]}%
\end{align*}
Now it is not difficult to see that $W_{h_{12}^{1}}(\gamma_{1})=W_{h_{12}^{2}%
}(\gamma_{2})$, $W_{h_{23}^{2}}(\gamma_{2})=W_{h_{23}^{3}}(\gamma_{3})$ and
$W_{h_{31}^{3}}(\gamma_{3})=W_{h_{31}^{1}}(\gamma_{1})$ exactly when there
exists $\gamma\in W_{G}$ with $\gamma_{i}=W_{k_{i}}(\gamma)$ ($i=1,2,3$), in
which $\gamma$ should uniquely be of the following form:
\begin{align*}
&  \gamma(X_{1},X_{2},X_{3},X_{4},X_{5},X_{6},X_{7},X_{8})\\
&  =a^{1}+a_{1}^{1}X_{1}+a_{2}^{1}X_{2}+a_{3}^{1}X_{3}+a_{6}^{1}X_{4}%
+a_{6}^{2}X_{5}+a_{6}^{3}X_{6}+a_{7}^{1}X_{7}+a_{7}^{2}X_{8}+a_{12}^{1}%
X_{1}X_{2}\\
&  +a_{13}^{1}X_{1}X_{3}+a_{16}^{1}X_{1}X_{4}+(a_{23}^{2}+a_{6}^{1})X_{2}%
X_{3}+a_{26}^{2}X_{2}X_{5}+a_{36}^{3}X_{3}X_{6}%
\end{align*}
This completes the proof of the theorem.
\end{proof}

\begin{corollary}
For any $\gamma_{123}$, $\gamma_{132}$, $\gamma_{213}$, $\gamma_{231}$,
$\gamma_{312}$, $\gamma_{321}\in M\otimes W_{D^{3}}$, if all expressions
(2.9)-(2.11) are well defined, then there exists unique $\gamma\in M\otimes
W_{G}$ such that
\begin{align*}
(\mathrm{id}_{M}\otimes W_{k_{1}})(\gamma)  &  =h_{(\gamma_{321},\gamma
_{231},\gamma_{132},\gamma_{123})}^{1}\\
(\mathrm{id}_{M}\otimes W_{k_{2}})(\gamma)  &  =h_{(\gamma_{132},\gamma
_{312},\gamma_{213},\gamma_{231})}^{2}\\
(\mathrm{id}_{M}\otimes W_{k_{3}})(\gamma)  &  =h_{(\gamma_{213},\gamma
_{123},\gamma_{321},\gamma_{312})}^{3}%
\end{align*}

\end{corollary}

\begin{remark}
This means that $\gamma$ encodes $\gamma_{123}$, $\gamma_{132}$, $\gamma
_{213}$, $\gamma_{231}$, $\gamma_{312}$ and $\gamma_{321}$. We can decode
$\gamma$, by way of example, into $\gamma_{123}$ via $\mathrm{id}_{M}\otimes
W_{k_{1}\circ\iota_{4}^{1}}$ or $\mathrm{id}_{M}\otimes W_{k_{3}\circ\iota
_{2}^{3}}$.
\end{remark}

\begin{proof}
(of the Corollary) Since
\begin{align*}
(\mathrm{id}_{M}\otimes W_{h_{12}^{1}})(h_{(\gamma_{321},\gamma_{231}%
,\gamma_{132},\gamma_{123})}^{1})  &  =l_{(\gamma_{231},\gamma_{132}%
)}=(\mathrm{id}_{M}\otimes W_{h_{12}^{2}})(h_{(\gamma_{132},\gamma
_{312},\gamma_{213},\gamma_{231})}^{2})\text{,}\\
(\mathrm{id}_{M}\otimes W_{h_{23}^{2}})(h_{(\gamma_{132},\gamma_{312}%
,\gamma_{213},\gamma_{231})}^{2})  &  =l_{(\gamma_{312},\gamma_{213}%
)}=(\mathrm{id}_{M}\otimes W_{h_{23}^{3}})(h_{(\gamma_{213},\gamma
_{123},\gamma_{321},\gamma_{312})}^{3})\text{ and}\\
(\mathrm{id}_{M}\otimes W_{h_{31}^{3}})(h_{(\gamma_{213},\gamma_{123}%
,\gamma_{321},\gamma_{312})}^{3})  &  =l_{(\gamma_{123},\gamma_{321}%
)}=(\mathrm{id}_{M}\otimes W_{h_{31}^{1}})(h_{(\gamma_{321},\gamma
_{231},\gamma_{132},\gamma_{123})}^{1})\text{,}%
\end{align*}
the desired conclusion follows directly from the above theorem.
\end{proof}

\begin{notation}
We will write $m_{(\gamma_{123},\gamma_{132},\gamma_{213},\gamma_{231}%
,\gamma_{312},\gamma_{321})}$ or $m\ $for short for the above $\gamma$.
\end{notation}

Once the above theorem is established, we can easily establish the general
Jacobi identity as follows:

\begin{proof}
(of the general Jacobi identity) Indeed we note that for any $d\in D$, we
have
\begin{align*}
&  (\gamma_{123}\overset{\cdot}{\underset{1}{-}}\gamma_{132})\overset{\cdot
}{-}(\gamma_{231}\overset{\cdot}{\underset{1}{-}}\gamma_{321})\\
&  =(\mathrm{id}_{M}\otimes W_{d\in D\mapsto(0,0,0,0,0,0,d)\in E[1]}%
)(h_{(\gamma_{321},\gamma_{231},\gamma_{132},\gamma_{123})}^{1})\\
&  =(\mathrm{id}_{M}\otimes W_{d\in D\mapsto(0,0,0,0,0,0,d)\in E[1]}%
)\circ(\mathrm{id}_{M}\otimes W_{k_{1}})(m)\\
&  =(\mathrm{id}_{M}\otimes W_{d\in D\mapsto(0,0,0,0,0,0,d,0)\in G})(m)\\
&  (\gamma_{231}\overset{\cdot}{\underset{2}{-}}\gamma_{213})\overset{\cdot
}{-}(\gamma_{312}\overset{\cdot}{\underset{2}{-}}\gamma_{132})\\
&  =(\mathrm{id}_{M}\otimes W_{d\in D\mapsto(0,0,0,0,0,0,d)\in E[2]}%
)(h_{(\gamma_{132},\gamma_{312},\gamma_{213},\gamma_{231})}^{2})\\
&  =(\mathrm{id}_{M}\otimes W_{d\in D\mapsto(0,0,0,0,0,0,d)\in E[2]}%
)\circ(\mathrm{id}_{M}\otimes W_{k_{2}})(m)\\
&  =(\mathrm{id}_{M}\otimes W_{d\in D\mapsto(0,0,0,0,0,0,0,d)\in G})(m)\\
&  (\gamma_{312}\overset{\cdot}{\underset{3}{-}}\gamma_{321})\overset{\cdot
}{-}(\gamma_{123}\overset{\cdot}{\underset{3}{-}}\gamma_{213})\\
&  =(\mathrm{id}_{M}\otimes W_{d\in D\mapsto(0,0,0,0,0,0,d)\in E[3]}%
)(h_{(\gamma_{213},\gamma_{123},\gamma_{321},\gamma_{312})}^{3})\\
&  =(\mathrm{id}_{M}\otimes W_{d\in D\mapsto(0,0,0,0,0,0,d)\in E[3]}%
)\circ(\mathrm{id}_{M}\otimes W_{k_{3}})(m)\\
&  =(\mathrm{id}_{M}\otimes W_{d\in D\mapsto(0,0,0,0,0,0,-d,-d)\in G})(m)
\end{align*}
Therefore, letting $t_{1}$, $t_{2}$ and $t_{3}$ denote the three expressions
in Theorem \ref{t5.8} in order, we have
\[
l_{(t_{1},t_{2},t_{3})}=(\mathrm{id}_{M}\otimes W_{(d_{1},d_{2},d_{3})\in
D(3)\mapsto(0,0,0,0,0,0,d_{1}-d_{3},d_{2}-d_{3})\in G})(m)
\]
This means that
\begin{align*}
&  t_{1}+t_{2}+t_{3}\\
&  =(\mathrm{id}_{M}\otimes W_{d\in D\mapsto(d,d,d)\in D(3)})(l_{(t_{1}%
,t_{2},t_{3})})\\
&  =(\mathrm{id}_{M}\otimes W_{d\in D\mapsto(d,d,d)\in D(3)})\circ
(\mathrm{id}_{M}\otimes W_{(d_{1},d_{2},d_{3})\in D(3)\mapsto
(0,0,0,0,0,0,d_{1}-d_{3},d_{2}-d_{3})\in G})(m)\\
&  =(\mathrm{id}_{M}\otimes W_{d\in D\mapsto(0,0,0,0,0,0,d-d,d-d)\in G})(m)\\
&  =(\mathrm{id}_{M}\otimes W_{d\in D\mapsto(0,0,0,0,0,0,0,0)\in G})(m)
\end{align*}
This completes the proof of the general Jacobi identity.
\end{proof}

\section{From the General Jacobi Identity to the Jacobi Identity of Lie
brackets}

\begin{theorem}
For any $X,Y\in\aleph(M)$, we have
\[
\lbrack X,Y]=Y\ast X\overset{\cdot}{-}(\mathrm{id}_{M^{M}}\otimes
W_{(d_{1},d_{2})\in D^{2}\mapsto(d_{2},d_{1})\in D^{2}})(X\ast Y)
\]

\end{theorem}

\begin{proof}
Let the putative mapping $i:D^{3}\{(1,3),(2,3)\}\rightarrow D^{3}$ be the
canonical embedding. Since
\begin{align*}
&  (\mathrm{id}_{M}\otimes W_{\varphi})(\mathrm{id}_{M^{M}}\otimes
W_{(d_{1},d_{2},d_{3})\in D^{3}\{(1,3),(2,3)\}\mapsto(d_{2},d_{3},d_{1})\in
D^{3}})(X\ast\lbrack X,Y]\ast Y)\\
&  =(\mathrm{id}_{M}\otimes W_{(d_{1},d_{2})\in D^{2}\mapsto(d_{2},0,d_{1})\in
D^{3}})(X\ast\lbrack X,Y]\ast Y)\\
&  =(\mathrm{id}_{M^{M}}\otimes W_{(d_{1},d_{2})\in D^{2}\mapsto(d_{2}%
,d_{1})\in D^{2}})(X\ast Y)
\end{align*}
and
\begin{align*}
&  (\mathrm{id}_{M}\otimes W_{\psi})(\mathrm{id}_{M^{M}}\otimes W_{(d_{1}%
,d_{2},d_{3})\in D^{3}\{(1,3),(2,3)\}\mapsto(d_{2},d_{3},d_{1})\in D^{3}%
})(X\ast\lbrack X,Y]\ast Y)\\
&  =(\mathrm{id}_{M}\otimes W_{(d_{1},d_{2})\in D^{2}\mapsto(d_{2},d_{1}%
d_{2},d_{1})\in D^{3}})(X\ast\lbrack X,Y]\ast Y)\\
&  =(\mathrm{id}_{M}\otimes W_{(d_{1},d_{2})\in D^{2}\mapsto(d_{2}%
,-d_{2},d_{1},d_{2},-d_{1},d_{1})\in D^{6}})(X\ast X\ast Y\ast X\ast Y\ast
Y)\\
&  =(\circ_{M^{M}}\otimes\mathrm{id}_{W_{D^{2}}})((\mathrm{id}_{M}\otimes
W_{(d_{1},d_{2})\in D^{2}\mapsto(d_{2},-d_{2})\in D^{2}})(Y\ast Y),Y\ast X,\\
&  (\mathrm{id}_{M}\otimes W_{(d_{1},d_{2})\in D^{2}\mapsto(-d_{1},d_{1})\in
D^{2}})(X\ast X))\\
&  \text{[by the bifunctionality of }\otimes\text{]}\\
&  =(\circ_{M^{M}}\otimes\mathrm{id}_{W_{D^{2}}})(I_{2},Y\ast X,I_{2})\\
&  =Y\ast X
\end{align*}
we have
\begin{align*}
&  Y\ast X\overset{\cdot}{-}(\mathrm{id}_{M^{M}}\otimes W_{(d_{1},d_{2})\in
D^{2}\mapsto(d_{2},d_{1})\in D^{2}})(X\ast Y)\\
&  =(\mathrm{id}_{M}\otimes W_{d\in D\mapsto(0,0,d)\in D^{3}\{(1,3),(2,3)\}}%
)(\mathrm{id}_{M^{M}}\otimes W_{(d_{1},d_{2},d_{3})\in D^{3}%
\{(1,3),(2,3)\}\mapsto(d_{2},d_{3},d_{1})\in D^{3}})\\
&  (X\ast\lbrack X,Y]\ast Y)\\
&  =(\mathrm{id}_{M^{M}}\otimes W_{d\in D\mapsto(0,d,0)\in D^{3}}%
)(X\ast\lbrack X,Y]\ast Y)\\
&  =[X,Y]
\end{align*}
Thus we are done.
\end{proof}

The following proposition should be obvious.

\begin{proposition}
For any $X\in\aleph(M)$ and any $\gamma_{1},\gamma_{2}\in\aleph^{2}(M)$ with
$(\mathrm{id}_{M^{M}}\otimes W_{i_{D(2)}^{D^{2}}})(\gamma_{1})=(\mathrm{id}%
_{M^{M}}\otimes W_{i_{D(2)}^{D^{2}}})(\gamma_{2})$, we have the following:
\begin{align*}
\gamma_{1}\ast X|_{D^{3}\{(2,3)\}}  &  =\gamma_{2}\ast X|_{D^{3}\{(2,3)\}}\\
X\ast\gamma_{1}|_{D^{3}\{(1,2)\}}  &  =X\ast\gamma_{2}|_{D^{3}\{(1,2)\}}\\
\gamma_{1}\ast X\overset{\cdot}{\underset{1}{-}}\gamma_{2}\ast X  &
=(\gamma_{1}\overset{\cdot}{-}\gamma_{2})\ast X\\
X\ast\gamma_{1}\overset{\cdot}{\underset{3}{-}}X\ast\gamma_{2}  &
=(\mathrm{id}_{M^{M}}\otimes W_{\tau_{(12)}})(X\ast(\gamma_{1}\overset{\cdot
}{-}\gamma_{2}))
\end{align*}

\end{proposition}

\begin{theorem}
For any $X,Y,Z\in\aleph(M)$, we have
\[
\lbrack X,[Y,Z]]+[Y,[Z,X]]+[Z,[X,Y]]=0
\]

\end{theorem}

\begin{proof}
we define $\gamma_{123},\gamma_{132},\gamma_{213},\gamma_{231},\gamma
_{312},\gamma_{321}\in\aleph^{3}(M)$ as follows:
\begin{align*}
\gamma_{123}  &  =Z\ast Y\ast X\\
\gamma_{132}  &  =(\mathrm{id}_{M^{M}}\otimes W_{(d_{1},d_{2},d_{3})\in
D^{3}\mapsto(d_{1},d_{3},d_{2})\in D^{3}})(Y\ast Z\ast X)\\
\gamma_{213}  &  =(\mathrm{id}_{M^{M}}\otimes W_{(d_{1},d_{2},d_{3})\in
D^{3}\mapsto(d_{2},d_{1},d_{3})\in D^{3}})(Z\ast X\ast Y)\\
\gamma_{231}  &  =(\mathrm{id}_{M^{M}}\otimes W_{(d_{1},d_{2},d_{3})\in
D^{3}\mapsto(d_{2},d_{3},d_{1})\in D^{3}})(X\ast Z\ast Y)\\
\gamma_{312}  &  =(\mathrm{id}_{M^{M}}\otimes W_{(d_{1},d_{2},d_{3})\in
D^{3}\mapsto(d_{3},d_{1},d_{2})\in D^{3}})(Y\ast X\ast Z)\\
\gamma_{321}  &  =(\mathrm{id}_{M^{M}}\otimes W_{(d_{1},d_{2},d_{3})\in
D^{3}\mapsto(d_{3},d_{2},d_{1})\in D^{3}})(X\ast Y\ast Z)
\end{align*}
Then the right-hand sides of the following three identities are meaningful,
and all the three identities hold:
\begin{align*}
\lbrack X,[Y,Z]]  &  =(\gamma_{123}\overset{\cdot}{\underset{1}{-}}%
\gamma_{132})\overset{\cdot}{-}(\gamma_{231}\overset{\cdot}{\underset{1}{-}%
}\gamma_{321})\\
\lbrack Y,[Z,X]]  &  =(\gamma_{231}\overset{\cdot}{\underset{2}{-}}%
\gamma_{213})\overset{\cdot}{-}(\gamma_{312}\overset{\cdot}{\underset{2}{-}%
}\gamma_{132})\\
\lbrack Z,[X,Y]]  &  =(\gamma_{312}\overset{\cdot}{\underset{3}{-}}%
\gamma_{321})\overset{\cdot}{-}(\gamma_{123}\overset{\cdot}{\underset{3}{-}%
}\gamma_{213})
\end{align*}
Therefore the desired result follows from the general Jacobi identity.
\end{proof}

\end{document}